\documentclass[11pt]{article}
\usepackage{latexsym}
\usepackage{mathbbol}
\usepackage{amssymb}
\usepackage{amsmath}
\usepackage{amsrefs}
\usepackage[all]{xy}
\usepackage{amsthm}
\usepackage{enumerate}
\usepackage{color}

\usepackage{tikz}

\newtheorem{teo}{Theorem}[section]
\newtheorem{prop}[teo]{Proposition}
\newtheorem{lemma}[teo] {Lemma}

\newtheorem{example}[teo]{Example}
\newtheorem{remark}[teo]{Remark}
\newtheorem{definition}[teo]{Definition}

\let\Im\relax

\DeclareMathOperator{\Hom}{Hom}
\DeclareMathOperator{\Im}{Im}
\DeclareMathOperator{\Min}{min}
\DeclareMathOperator{\Max}{max}

\DeclareMathOperator{\Sub}{Sub}
\DeclareMathOperator{\Ext}{Ext}
\DeclareMathOperator{\Ker}{Ker}
\DeclareMathOperator{\HH}{HH}

\DeclareMathOperator{\id}{id}
\DeclareMathOperator{\Id}{Id}

\renewcommand{\k} {\mathbb k}
\newcommand{\R}{\mathcal R}
\renewcommand{\P}{\mathcal P}

\newcommand{\Ap}{\mathbb {Ap}}
\renewcommand{\Bar}{\mathbb {Bar}}

\begin{document}

\title{Comparison morphisms between two projective resolutions of monomial algebras}
\date{}
\author{Mar\'\i a Julia Redondo and Lucrecia Rom\'an \footnote{Instituto de Matem\'atica, Departamento de Matem\'atica, 
Universidad Nacional del Sur (UNS)-CONICET, Bah\'\i a Blanca -  Argentina.
{\it E-mail address:  mredondo@criba.edu.ar, lroman@uns.edu.ar}}
\thanks{The first author is a researcher and the second author has a fellowship at CONICET, Argentina. This work has been supported by the project PICT-2011-1510.  }}

\maketitle

\begin{abstract}
We construct comparison morphisms between two well-known projective resolutions of a monomial algebra $A$: the bar resolution $\Bar A$ and  Bardzell's resolution $\Ap A$; the first one is used to define the cup product and the Lie bracket on the Hochschild cohomology $\HH^*(A)$ and the second one has been shown to be an efficient tool for computations of these cohomology groups. The constructed comparison morphisms allow us to show that the cup product restricted to  even degrees of the Hochschild cohomology has a very simple description. Moreover, for $A=\k Q/I$ a monomial algebra such that $\dim_\k e_i A e_j = 1$ whenever there exists an arrow $\alpha: i \to j \in Q_1$, we describe the Lie action of the Lie algebra $\HH^1(A)$  on $\HH^{\ast}(A)$. 
\end{abstract}

\noindent 2010 MSC: 16E40 16W99

\section{Introduction}

Hochschild cohomology of associative algebras was introduced by G. Hochschild in~\cite{Ho} and since then it has been studied by many authors in different areas of mathematics. In the case where A is an algebra over a ring  $\k$ such that $A$ is $\k$-projective, Cartan and Eilenberg give  a useful interpretation of  the Hochschild cohomology groups ${\HH}^n(A,A)$ with coefficients in $A$. They prove that these groups can be identified  with the groups $\Ext_{A^e}^n(A,A)$  and thus be calculated using  arbitrary projective resolutions of $A$ over its enveloping algebra $A^e$,  see~\cite{CE}.  Despite this freedom of choice to calculate cohomology groups,  it is not the case with some of the structures defined in cohomology: the sum ${\HH}^*(A) = \oplus_{n \geq 0} {\HH}^n(A)$ is a Gerstenhaber algebra, that is, it is a graded
commutative ring via the cup product $\cup: \HH^n(A) \times \HH^m(A) \to \HH^{n+m}(A)$, a graded Lie algebra via the bracket $[ -  ,  - ]: \HH^n(A) \times \HH^m(A) \to \HH^{n+m-1}(A)$, and these
two structures are related, see~\cite{G}.  One wants to understand the structure of  ${\HH}^*(A)$ as a ring and as graded Lie algebra but the problem is that these structures are defined in terms of the bar resolution $\Bar A$ of $A$, where historically the cohomology groups were defined, and the computations of these groups are made, in general, using a convenient projective resolution. 

Important improvements have been made when considering the cup product: it has another description, the Yoneda product, which can be transported easily to the complex obtained by using any other projective resolution.  This has been used to described the ring structure of $\HH^*(A)$ for many algebras $A$ such as radical square zero algebras~\cite{C3}, truncated quiver algebras~\cite{ACT}, Koszul algebras~\cite{BGSS} and so on.  

On the other hand, so far, the bracket has no similar description; for this reason there are only a few classes of algebras in which the bracket has been determined.  The question about finding a way to compute the Gerstenhaber bracket on Hochschild cohomology in terms of arbitrary projective resolutions, was raised by  Gerstenhaber and Schack in~\cite{GS}*{p. 256}.  Ten years later,  Schwede gave in~\cite{SS} a beautiful interpretation of the bracket in terms of bimodule self-extensions of $A$, based on Retakh's description of extension categories.  Recently, Negron and Witherspoon succeeded in finding, under some conditions, a definition of   Gerstenhaber's graded Lie bracket on complexes other than the bar complex (see~\cite{NW}) and  Su\'arez-Alvarez  gave a way to compute the Lie brackets restricted to $\HH^1(A) \times  \HH^{\ast}(A)$ in terms of  arbitrary projective resolutions of $A$, see~\cite{MSA}. 

In this paper we concentrate on the  particular case of monomial algebras, that is, algebras $A=\k Q/I$ where  $I$ can be chosen as generated by paths. For this family of algebras one has a detailed description of a minimal resolution of the $A^e$-module $A$ given by Bardzell in~\cite{B};  we will denote it by $\Ap A$ and recall it  in Section 2. Although  the construction of this resolution has hard combinatorial calculations,  it has been shown to be an efficient
tool for many computations because it can be described directly from paths in the quiver $Q$. Then it is natural to try to define both structures, the cup product and the Lie bracket, using this resolution.
 In this article we construct explicit comparison morphisms  
 $$ \xymatrix{  
{\Bar A}  \ar@<.9ex>[rr]^{\quad {\mathbf G}}&  &    \ar@<.9ex>[ll]^{\quad {\mathbf F}}  { \Ap A}
} $$  between the two resolutions involved. Thus using the quasi-isomorphisms between the derived complexes we can transport   structures on the complex $\Hom_{A^e}({ \Bar A},A)$ to the complex $\Hom_{A^e}({ \Ap A},A)$ and  as an application we are able to define the Gerstenhaber algebra structures  on the complex obtained using  Bardzell's minimal resolution.

When restricting to truncated quiver algebras, a subclass of monomial algebras, our comparison morphisms $\mathbf F$ and $\mathbf G$ restrict to those defined in \cite{ACT}.

The paper is organized as follows. In Section 2 we introduce all the necessary terminology,  we recall the two projective resolutions and the  definition of a Gerstenhaber algebra. In Section 3 we define the maps ${\mathbf F}: \Ap A \to \Bar A$ and ${\mathbf G}: \Bar A \to \Ap A$, and we establish the main result of this work: for a monomial algebra $A$, $\mathbf F$ and $\mathbf G$ are comparison morphisms.
We also present some  properties of these morphisms for the subsequent proof in Sections 4 and 5. The last section provides a small
example to show our  technique together with two more general applications: in the first one, we study the Lie action of the Lie algebra $\HH^1(A)$ of outer derivations of $A$ on $\HH^{\ast}(A)$, when $A$ is a monomial algebra that satisfies the following property: $\dim_\k e_i A
e_j = 1$ if there exists $\alpha: i \to j \in Q_1$. In the second one, we show that the cup product restricted to even degrees of the
Hochschild cohomology has a very simple description.

\section{Preliminaries}

\subsection{Quivers, relations and monomial algebras}

We briefly recall some concepts concerning quivers and monomial algebras; for unexplained notions we refer the reader, for instance, to \cite{ASS}.

A finite quiver $Q$ is a finite set of vertices $Q_0$, a finite set of arrows $Q_1$, and two maps  $s, t : Q_1 \to Q_0$ associating to each arrow $\alpha$ its source  $s(\alpha)$ and its target $t(\alpha)$.  A path $w$ of length $l$ is a sequence of $l$ arrows $\alpha_1 \dots \alpha_l$ such that $t(\alpha_i)=s(\alpha_{i+1})$. We denote by $\vert w \vert $ the length of the path $w$. We put $s(w)=s(\alpha_1)$ and $t(w)=t(\alpha_l)$. For any vertex $x$ we consider $e_x$ the trivial path of length zero and we put $s(e_x)=t(e_x)=x$. 

We say that a path $w$ divides a path $u$ if $u = L(w) w R(w)$, where $L(w)$ and $R(w)$ are not simultaneously paths of length zero.

The path algebra $\k Q$ is the $\k$-vector space with basis the set of paths in $Q$; the product on the basis elements is given by the concatenation of the sequences of arrows of the paths $w$ and $w'$ if they form a path (namely, if $t(w)=s(w')$) and zero otherwise.  Vertices form a complete set of orthogonal idempotents of $\k Q$. Let $F$ be the two-sided ideal of $\k Q$ generated by the arrows of $Q$. A two-sided ideal $I$ of $\k Q$ is said to be \textit{admissible} if there exists an integer $m \geq 2$ such that $F^m \subseteq I \subseteq F^2$. The elements in $I$ are called \textit{relations}, and $\k Q/I$ is called a \textit{monomial algebra} if the ideal $I$ is generated by paths.

By a fundamental result in representation theory  it  is well known that if $A$ is  an associative,  basic, indecomposable, finite dimensional algebra over an algebraically closed field $\k$,   there exists a finite quiver $Q$ such that $A$ is Morita equivalent to $\k Q/I$, where $\k Q$ is the path algebra of $Q$ and $I$ is an admissible two-sided ideal of $\k Q$.

 From now on we will assume that $A = \k Q/I$  is a monomial algebra. We also  assume that the ideal $I$ is generated by paths of minimal length, and we fix a minimal set $\R$ of paths, of minimal length, that generate the ideal $I$.  Moreover, we denote by $\P$ the set of paths in $Q$ such that the set $\{ \gamma + I, \gamma \in \P\} $ is a basis of $A=\k Q/I$. It is clear that $Q_0 \cup Q_1 \subseteq \P$ since $I \subseteq F^2$.

\subsection{The standard bar resolution $\Bar A$}

The bar resolution $\Bar A=(B_n, b_{n+1})_{n \geq 0}$ is the following resolution of $A$ by $A^e$-modules, where $A^e=A \otimes A^{op}$ is the enveloping algebra and $\otimes = \otimes_\k$.  To begin with, $B_n=A^{\otimes^{(n+2)}}$ is the $(n+2)$-fold tensor product of $A$ with itself over $\k$.  The  $A^e$-linear map $b_{n+1}: B_{n+1} \to B_{n}$,
$$b_{n+1}(\alpha_0\otimes \alpha_1\otimes \dots \otimes\alpha_{n+1}\otimes \alpha_
{n+2})  =  \sum_{i=0}^{n+1} (-1)^i \alpha_0\otimes
\alpha_1\otimes \dots \otimes
\alpha_i\alpha_{i+1}\otimes \dots \otimes\alpha_{n+1}\otimes \alpha_{n+2},$$
turns $\Bar A$ into a complex, which is acyclic in all degrees except in degree $0$, wherein its homology is isomorphic to $A$. The multiplication map $\varepsilon: A \otimes A \to A$ given by $\varepsilon (a\otimes b)=ab$  provides an augmentation $\Bar A \to A \to 0$.

For a path algebra $A = \k Q/I$, we can consider $E=\k Q_0$ the subalgebra of $A$ generated by the set of vertices $Q_0$, and in this case the standard bar resolution can be redefined using tensor products over $E$. We also denote this resolution by  $\Bar A$.  

It is a very well-known fact that the Hochschild cohomology $\HH^\ast (A)$ is isomorphic to $H^\ast (\Hom_{A^e}(\Bar A, A))$.

\subsection{Bardzell's resolution $\Ap A$}

Bardzell's resolution $\Ap A= (A \otimes \k AP_n \otimes A, d_{n+1})_{n \geq 0}$ is a minimal resolution that was introduced by Bardzell in \cite{B} for monomial algebras.

Given a monomial algebra $A=\k Q/I$ with $\R$ a minimal set of paths, of minimal length, that generate the ideal $I$,  let $AP_0=Q_0$, $AP_1=Q_1$ and for $n\geq 2$ let  $AP_n$ be the set of supports of  {\it $n$-concatenations} which are  defined inductively as follows: given any directed path $T$ in $Q$, consider the set of vertices that are starting and ending points of arrows belonging to $T$, and consider the natural order $<$ in this set.  Let $\R(T)$ be the set of paths in $\R$ that are contained in the directed path $T$. Take $p_1 \in \R(T)$ and consider the set $$L_1=\{ p \in \R(T): s(p_1)  < s(p) < t(p_1)  \}.$$
If $L_1 \not = \emptyset$, let $p_2$ be such that $s(p_2)$ is minimal with respect to all $p \in L_1$. Now assume that $p_1, p_2, \dots, p_j$ have been constructed.  Let
$$L_{j+1}=\{ p \in \R(T): t(p_{j-1}) \leq s(p) < t(p_j) \}.$$
If $L_{j+1} \not = \emptyset$, let $p_{j+1}$ be such that $s(p_{j+1})$ is minimal with respect to all $p \in L_{j+1}$. Thus $(p_1, \dots, p_{n-1})$ is an $n$-concatenation and we denote by $w(p_1, \dots, p_{n-1})$ the {\it support} of the concatenation, that is, the path from $s(p_1)$ to $t(p_{n-1})$ along the directed path $T$.

These concatenations can be pictured as follows:
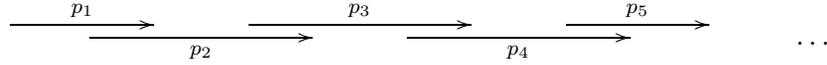
\[ \xymatrix{ \ar@<1ex>[rr]^{p_1} & \ar[rrr]_{p_2} &  & \ar@<1ex>[rrr]^{p_3} & & \ar[rrr]_{p_4} & & \ar@<1ex>[rr]^{p_5} & & & \dots  } \]
For any $w \in AP_n$ define $\Sub(w)= \{ w' \in AP_{n-1}: w' \ \mbox{divides} \ w\}$. 

\medskip

We can dualize the construction of the sets $AP_{n}$: given $q_1 \in \R(T)$ consider the set 
$$L_1^{op} =\{ q \in \R(T): s(q_1)  < t(q) < t(q_1)  \}.$$
If $L_1^{op} \not = \emptyset$, let $q_2$ be such that $t(q_2)$ is maximal with respect to all $q \in L_1^{op}$. Now assume that $q_1, q_2, \dots, q_j$ have been constructed.  Let
$$L_{j+1}^{op}=\{ q \in \R(T): s(q_{j}) < t(q) \leq s(q_{j-1}) \}.$$
If $L_{j+1}^{op} \not = \emptyset$, let $q_{j+1}$ be such that $t(q_{j+1})$ is maximal with respect to all $q \in L_{j+1}^{op}$.
Thus $(q_{n-1} , \dots, q_1)$ is an {\it $n$-op-concatenation}, we denote by $w^{op}(q_{n-1}, \dots, q_1)$ the support of the concatenation, that is,  the path from $s(q_{n-1})$ to $t(q_1)$ along the directed path $T$, and  $AP^{op}_{n}$ is the set of supports of $n$-op-concatenations. Moreover, we denote $w^{op}(q_{n-1}, \dots, q_1)=w^{op}(q^1, \dots, q^{n-1})$. It is shown in \cite{B}*{Lemma 3.1} that $AP_{n} = AP^{op}_{n}$.

\medskip

Now we are ready to describe Bardzell's resolution $\Ap A= (A \otimes \k AP_n \otimes A, d_{n+1})_{n \geq 0}$. To begin with, 
$\k X$ is the vector space generated by the set $X$ and all tensor products are taken over $E=\k Q_0$, the subalgebra of $A$ generated by the vertices.  In order to define the $A^e$-linear maps $d_n: A \otimes \k AP_{n} \otimes A \to A \otimes  \k AP_{n-1} \otimes A$ we need the following notations: if $n \geq 2$, for any $w \in AP_{n}$ and $\psi \in \Sub(w)$ we denote $w= L(\psi) \psi R(\psi)$. In particular, if $n=2m+1$, then 
 $\Sub(w) = \{ \psi_1, \psi_2 \}$ and $w= \psi_ 1 R(\psi_1) = L(\psi_2) \psi_2$, see \cite{B}*{Lemma 3.3}. Then
\begin{align*}
d_1( 1 \otimes \alpha \otimes 1) & =   \alpha \otimes e_{t(\alpha)} \otimes 1 - 1 \otimes e_{s(\alpha)} \otimes \alpha, \\
d_{2m} (1 \otimes w \otimes 1) & =  \sum_{\psi \in \Sub(w)} L(\psi) \otimes \psi \otimes R(\psi), \\
d_{2m+1} (1 \otimes w \otimes 1) & =   L(\psi_2) \otimes \psi_2 \otimes 1 -  1 \otimes \psi_1 \otimes R(\psi_1) .
\end{align*}
The multiplication map $\mu: A \otimes \k Q_0 \otimes A \to A$ given by $\mu( 1 \otimes e_i \otimes 1 )  =  e_i$, provides an augmentation $\Ap A \to A \to 0$.

We are interested in algebras that are projective over $\k$, and in this case it is well known that the Hochschild cohomology $\HH^\ast (A)$ is isomorphic to $\Ext^\ast_{A^e}(A,A) = H^\ast(\Hom_{A^e}(\Ap A, A))$.

\subsection{The Gerstenhaber algebra $\HH^{\ast}(A)$}\label{productos}

In \cite{G} Gerstenhaber introduced  two structures on the Hochschild cohomology $\HH^{\ast}(A)$, namely the cup product $\cup$  and  the bracket $[-,-]$. They  are defined using explicit formulas in terms of cochains in the complex $\Hom_{A^e}({\Bar A}, A)$ as follows: given  $f \in \Hom_{A^e}(
A^{\otimes^{n+2}} , A)$ and $g \in \Hom_{A^e}( A^{\otimes^{m+2}}  , A)$ we have
$f \cup g \in \Hom_{A^e}(  A^{\otimes^{m+n+2}}  , A)$  defined by
$$  f \cup g (1\otimes  v_1 \otimes \dots \otimes v_{n+m} \otimes 1 ) = f( 1 \otimes v_1 \otimes \dots \otimes v_{n} \otimes 1  )g(1 \otimes v_{n+1} \otimes \dots \otimes v_{n+m} \otimes 1  ) $$
and $ [f,g] \in \Hom_{A^e}(A^{\otimes^{m+n+1}} , A)$ defined by
$$ [f,g ]= f \circ g - (-1)^{(n-1)(m-1)} g \circ f
$$where
\begin{equation*} f \circ g = \sum_{i=1}^n(-1)^{(i-1)(m-1)}f \circ_i g 
\end{equation*} and
\begin{align*}   f & \circ_i g  (1 \otimes v_1 \otimes \dots \otimes v_{n+m-1} \otimes 1   ) \\ & =  f (1 \otimes v_1 \otimes
\dots   \otimes v_{i-1} \otimes   g( 1 \otimes v_i \otimes \dots v_{i+m-1} \otimes 1 ) \otimes v_{i+m} \otimes
\dots v_{n+m-1} \otimes 1 ).\end{align*}
These   products  induce well defined products on Hochschild cohomology
\begin{align*} \cup & :  \HH^n(A) \times \HH^m(A) \to \HH^{n+m}(A) \\
[ - , - ] & :  \HH^n(A) \times \HH^m(A) \to \HH^{n+m-1}(A)  \end{align*}
in such a way that $(\HH^{\ast}(A), \cup, [ - , - ])$ becomes a Gerstenhaber algebra, that is, 
$(\HH^*(A), \cup)$  is a graded
commutative ring, $(\HH^*(A), [-,-])$ is a graded Lie algebra, and the bracket is compatible with the cup product since it acts through graded derivations, see~\cite{G}. 

\section{The comparison morphisms}

A comparison morphism between two projective resolutions of an algebra $A$ is a morphism of chain complexes that lifts the identity map on $A$. The existence of such morphisms is clear,  see for example \cite{CE}.  However, an explicit construction of these morphisms is not  always easy. In the next two subsections we will define maps $${\mathbf F}: \Ap A \to \Bar A \quad  \mbox{and}  \quad {\mathbf G}: \Bar A \to \Ap A$$ that allow us to obtain  the main result of this article, that is, for a monomial algebra $A$ the maps ${\mathbf F}$ and  ${\mathbf G}$ are comparison morphisms.

We will start with the definition of the $A^e$-linear maps:
\[ \xymatrix{
A \otimes_E \k AP_n \otimes_E A   \ar@<.9ex>[rr]^{\qquad F_n}& & \ar@<.9ex>[ll]^{\qquad G_n}
A^{\otimes_E ^{n+2}} }\] 
for $n \geq 0$ and then we will show  the commutativity of the diagrams   
   \[ \xymatrix{
 A \otimes \k AP_{n} \otimes  A  \ar^{F_n}[d] \ar^{d_n}[r] &  A \otimes \k AP_{n-1} \otimes  A \ar^{F_{n-1}}[d]  &   A^{\otimes^{n+2}}   \ar^{G_n}[d] \ar^{b_n}[r] &   A^{\otimes^{n+1}}  \ar^{G_{n-1}}[d]  \\
   A^{\otimes^{n+2}}    \ar^{b_n}[r] &   A^{\otimes^{n+1}},  &   A \otimes \k AP_{n} \otimes  A \ar^{d_n}[r] &  A \otimes \k  AP_{n-1} \otimes A .} \]
This proof is not immediate, and Sections 4 and 5 are devoted exclusively to it. 

\subsection{The map ${\mathbf F}: \Ap A \to \Bar A$}\label{F}

We define the $A^e$-linear maps $F_n : A \otimes \k AP_n \otimes A \longrightarrow  A^{\otimes ^{n+2}}$ as follows:
\begin{align*}  F_0(1
\otimes e \otimes 1) & =  e \otimes 1, \\
F_1(1\otimes \alpha \otimes1) & =  1\otimes \alpha\otimes1, \\
F_n(1\otimes w\otimes 1) &  =   \sum_{i= 1}^{m-1} 1\otimes
L_{i+1}F_{n-1}(1\otimes \zeta_{i+1} \otimes 1)R_{i+1},
\mbox{ \quad if $n \geq 2$}
\end{align*}
where  $\Sub(w) = \{ \zeta_1, \dots, \zeta_m \} \subset
AP_{n-1}$ is an ordered set such that  if  $i <j$ then  $s(\zeta_i) <
s(\zeta_j) $ with respect to the order given in the support of $w$, and
$L_{i}, R_{i}$  are the paths defined by
 $$w = L_{i}\zeta_{i}R_{i} \qquad \mbox{for $i= 1 , \dots, m$. }$$ 
 
 \begin{remark}\label{obs1}
\begin{itemize} \item[ ]
 \item[1)]
If $n = 2$  and     $w = \alpha_1 \dots \alpha_s \in AP_2$ ,   $\alpha_i \in Q_1$,
then $\Sub(w) = \{ \alpha_1 ,  \dots ,\alpha_s \}$. Thus
\begin{align*}  F_2(1 \otimes w\otimes 1) &  =    \sum_{i=1}^{s-1} 1\otimes
\alpha_1 \dots \alpha_iF_1(1\otimes \alpha_{i+1} \otimes
1)\alpha_{i+2} \dots \alpha_s \\
&  = 
\sum_{i=1}^{s-1} 1\otimes \alpha_1 \dots\alpha_i\otimes \alpha_{i+1}
\otimes \alpha_{i+2} \dots \alpha_s. \end{align*}
\item[2)] If   $\Sub(w) =
\{\zeta_1,\zeta_2\}$ for $w \in AP_n $  with  $w = L(\zeta_2)\zeta_2$, we have that 
$$F_n(1\otimes w\otimes 1) = 1\otimes L(\zeta_2)F_{n-1}(1\otimes \zeta_2
\otimes 1).$$ In particular this is true if  $n$ is odd, see \cite{B}*{Lemma 3.3}.  
\item[3)] If $c \in A$ and $w \in AP_n$ then   $$b_{n+1}(1 \otimes cF_{n}(1
\otimes w \otimes 1)) = cF_n(1 \otimes w \otimes 1) - 1 \otimes
cb_nF_n(1 \otimes w \otimes 1),$$ since $b_{n+1}$ is linear and
\[ \begin{aligned} b_{n+1}(1 \otimes  c(a_0 \otimes \dots \otimes  a_{n+1} ))  & =  c(a_0 \otimes \dots \otimes  a_{n+1} ) - 1\otimes b_n(c(a_0 \otimes \dots \otimes  a_{n+1} )) \\
 & =  c(a_0 \otimes \dots \otimes  a_{n+1} ) - 1\otimes c b_n(a_0 \otimes \dots \otimes  a_{n+1} ).  \end{aligned} \]
\end{itemize}
\end{remark}

\subsection{The map ${\mathbf G}: \Bar A \to \Ap A$}\label{G}

Since the sought morphism $\mathbf{G} $ is  a  morphism  of $A$-bimodules, we only have to define it on basis elements  \[1\otimes v_1\otimes
\dots \otimes v_{n}\otimes1 , \mbox{ $v_i \in \P$, $t(v_i) =
s(v_{i+1})$}
\] of  $ A^{\otimes n+2} $. For this, we will  need to distinguish certain $n$-sequences $(v_1, \dots, v_n)$ in  $\P^n$.

\begin{definition}\label{mimp} An $n$-sequence $(v_1, \dots, v_n) \in \P^n$ is called  \textbf {well-concatenated} if  $t(v_i) = s(v_{i+1})$
for $i =1 , \dots,  n-1$. For any  well-concatenated $n$-sequence 
$(v_1, \dots, v_n)$ we define the sets
\begin{align*} \mathcal{M}_{odd}  =  \mathcal{M}_{odd}(v_1, \dots, v_n) & =  \{j: \  v_{2j-1}.v_{2j} \not \in I \},\\
 \mathcal{M}_{even}  = \mathcal{M}_{even}(v_1, \dots, v_n) & =  \{ j: \ v_{2j}.v_{2j+1} \not 
\in I \}. \end{align*} The  well-concatenated $n$-sequence  $(v_1,
\dots, v_n)$ is called \textbf{good} if  $n$ is even and  $\mathcal{M}_{odd} =
\emptyset$ or if $n$ is odd 
 and   $\mathcal{M}_{even} = \emptyset$. Otherwise, the  well-concatenated $n$-sequence  $(v_1, \dots, v_n)$  is called \textbf{bad}.
\end{definition}

For any  good $n$-sequence $(v_1, \dots, v_n)$, we consider the subset of  $AP_n$: \[\chi(v_1, \dots , v_n) = \{w \in
AP_n :  v_1\dots v_n = L(w) w R(w) \}.\] Now the
$A^e$-linear maps
$G_n :  A^{\otimes ^{n+2}} \longrightarrow A \otimes \k AP_n \otimes A$ are given by
\begin{align*}
G_0 (1 \otimes 1) & =   1 \otimes 1 \otimes 1 = \sum_{i \in Q_0} 1 \otimes e_i \otimes 1, \\
G_1 (1 \otimes v \otimes 1) & =  
\begin{cases}
\sum_{i=1}^s\alpha_1 \dots \alpha_{i-1}\otimes \alpha_i\otimes \alpha_{i+1} \dots \alpha_s , &  \hbox{if  $v= \alpha_1 \dots\alpha_s$,} \\
0,  &   \hbox{if $\vert v \vert =  0$}.
\end{cases}
\end{align*}
If   $1\otimes v_1\otimes \dots \otimes v_{n}\otimes1$ is a basis element in  $ A^{\otimes n+2}$ and the  $n$-sequence $(
v_1, \dots,  v_{n})$  is bad or  $\chi( v_1, \dots, v_{n})
= \emptyset$, then
\[G_n(1\otimes v_1\otimes \dots \otimes v_{n}\otimes1)
= 0.\] Otherwise,  if $\chi( v_1, \dots,  v_{n}) \neq \emptyset$,
\[G_{n} (1\otimes v_1\otimes \dots \otimes v_{n}\otimes1)   =
  L(w_1)\otimes w_1 \otimes R(w_1), \quad   \mbox{ if $n$ is even,} \] where  $w_1$ is such that  $s(w_1)= $ min $\{s(w) : w \in \chi( v_1, \dots,  v_{n}) \} $
and 
\[G_{n} (1\otimes v_1\otimes \dots \otimes v_{n}\otimes1)   =
\sum_{\substack{ w \in \chi( v_1, \dots,  v_{n}) \\ \\  s(v_1) \leq
s(w) < t(v_1) } } L(w)\otimes w \otimes R(w), \quad  \mbox{ if $n$
is odd}. \]

\section{The map $\mathbf F$ is a comparison morphism}\label{morfismoF}

First we establish some preliminary results about the sets $AP_n$ that will be used in the forthcoming proof. In the first  lemmas we describe right and left  divisors of paths of the form $aw$ and $wb$ respectively, for $w$ the support of a concatenation and  $a,b \in \P$.

\begin{lemma}
\label{A2} Let   $w=w^{op}(q^1, \dots, q^{2n-1}) \in AP_{2n}$.
\begin{itemize}
\item [(i)] If  $v=v(p_1, \dots, p_{2n-2}) \in AP_{2n-1}$ is such that   $ a  w  =  v  b   $ with  $a, b$ paths in $Q$ and  $a \in \mathcal{P}$  then $t(p_{2n-2}) \leq s(q^{2n-1})$, and therefore $b \in I$.
\item [(ii)] If $u=u(p_1, \dots, p_{2n-1}) \in AP_{2n}$ is such that  $a  w  =u  b   $ with $a, b \in \mathcal{P}$ then there  exists $z \in AP_{2n+1}$ such that   $z$ divides the path $ub$ and  $s(z)=s(u)$.
\end{itemize}
\end{lemma}
\begin{proof}
 (i) We use an inductive procedure to show that
\begin{align}
s(q^{2j-1}) < t (p_{2j-1} ) \label{x1}\\
s(p_{2j}) \leq s(q^{2j-1}) \label{x2} \\
t(p_{2j}) \leq t(q^{2j-1} )\label{x3}
\end{align}
for all $j= 1, \dots, n-1$. It is clear that the second inequality implies the third because  $p_{2j}, q^{2j-1}$ are  minimal relations. 
The hypothesis    $a \in \P$ implies that
\begin{equation}
s(p_1) < s(q^{1}) < t(p_{1}). \label{hipa}
\end{equation}
Since  $p_{2}$ has been chosen in the set
 $ L_1=  \{  \gamma \in \R(v) :  s(p_1) < s(\gamma) <
 t(p_1) 
\} $ with  $s(p_{2})$  minimal with respect to all 
 $\gamma \in L_1$, from \eqref{hipa} it follows that
$s(p_{2}) \leq s(q^{1})$  and therefore, $t(p_{2}) \leq t(q^{1})$.

By induction hypothesis we assume that the inequalities  \eqref{x1}, \eqref{x2} and  \eqref{x3} are satisfied. By
construction of  $v$ it follows that $t(p_{2j-1}) \leq s(p_{2j+1}) <
t(p_{2j})$ and using the inequalities of the inductive hypothesis we obtain that
$s(q^{2j-1}) < s(p_{2j+1}) < t(q^{2j-1})$. These last inequalities imply that \begin{equation} t(q^{2j-1})
< t(p_{2j+1}).\label{hipot3}\end{equation}  
Since   $q^{2j-1}$  has been chosen in the set
$ L_{2j-1}^{op}= \{ \gamma \in \R(w) : s(q^{2j}) < t(\gamma) \leq
 s(q^{2j+1}) 
\} $  with $t(q^{2j-1})$  maximum with respect to all   $\gamma \in  L_{2j-1}^{op}$,  \eqref{hipot3} implies that
 $p_{2j+1} \not \in  L_{2j-1}^{op}$, and since
$s(q^{2j}) < t(q^{2j-1}) < t(p_{2j+1})$ we have that
 \begin{equation} s(q^{2j+1}) < t(p_{2j+1}).\label{hipa4} \end{equation}
Therefore, from \eqref{x3},  \eqref{hipa4}  and by construction of
$w^{op}$ we have that
$$t(p_{2j}) \leq t(q^{2j-1}) \leq s(q^{2j+1}) <
t(p_{2j+1}).$$  Since $p_{2j+2}$ has been chosen
in the set  $ L_{2j+2} = \{ \gamma \in   \R(v) :  t(p_{2j}) \leq s(\gamma) < t(p_{2j+1}) 
\} $ with $s(p_{2j+2})$ minimal with respect to all
 $\gamma \in L_{2j+2}$, we can conclude that
 \begin{equation} s(p_{2j+2}) \leq s(q^{2j+1}) \quad
\mbox{and therefore}  \quad t(p_{2j+2}) \leq t(q^{2j+1}).
\end{equation}
In particular, we have that  $t(p_{2n-2}) \leq t(q^{2n-3}) \leq
s(q^{2n-1})$.

\medskip

(ii) To prove the existence of  $z$ we have to show that there exists    $p_{2n} \in \R$ such that  $z=z(p_1, \dots, p_{2n-1},p_{2n} )$ belongs to  $AP_{2n+1}$, that is,
we must verify that the set
$$\{ q \in \R(a w) : t(p_{2n-2}) \leq
s(q) < t(p_{2n-1}) \}$$ is not empty. Since $b \in \P$ we have that
$s(q^{2n-1})  < t(p_{2n-1})$, and from   (i) we know that $t(p_{2n-2})
\leq s(q^{2n-1})$. Therefore  $q^{2n-1}  \in \{ q \in \R(a  w) :
t(p_{2n-2}) \leq s(q) < t(p_{2n-1}) \}.$
\end{proof}
 \begin{lemma} \label{A}  Let
$w=w(p_1, \dots, p_{n-1}) \in AP_{2n}$.
\begin{itemize}
\item [(i)] If  $v=v^{op}(q^2, \dots, q^{2n-1}) \in AP_{2n-1}$ is such that  $w  b = a v  $ with $a, b$ paths in   $Q$ and $b \in \P$ then  $t(p_1) \leq
s(q^2)$, and therefore  $a \in I$.
\item [(ii)] If $u=u^{op}(q^1, \dots, q^{2n-1}) \in AP_{2n}$ is such that  $w b = a u $ with  $a, b$ paths in   $\P$ then there  exists $z \in AP_{2n+1}$ such that   $z$ divides the path   $au$ and $t(z)=t(u)$.
\end{itemize}
\end{lemma}

\begin{proof} The proof is dual to that of the previous lemma.
\end{proof}

\begin{prop} \label{cero 2}  Let  $w \in AP_{2n+1}$, $\Sub(w) = \{\psi_1, \psi_2\}$ with  $w =   \psi_1 R(\psi_1) =  L(\psi_2)\psi_2$. If $w = w(p_1,  \dots ,p_{2n}) = w^{op}(q^1, \dots ,q^{2n})$, we have: 
\begin{itemize}   
\item[(i)] If  $\gamma \in \Sub(\psi_2)$ is such that $t(\psi_1) < t(\gamma)$, then $t(p_1) \leq s(\gamma)$. 
\item[(ii)] If $\gamma \in \Sub(\psi_1)$ is such that  $s(\gamma) < s(\psi_2)$, then $t(\gamma) \leq s(q^{2n})$. 
\end{itemize} 
\end{prop} 
\begin{proof} 
\begin{itemize}   
\item[(i)]  Let  $a,b$ be such that   $  \psi_1 b = a \gamma $. Then $b \in \P$ because $b$ divides   $ R(\psi_1)$ and the result follows from Lemma \ref{A}(i).   \item[(ii)]  Similarly, it follows from  Lemma \ref{A2}(i). 
\end{itemize} 
\end{proof}

In the following lemma we show that we can weaken the assumptions on item $(ii)$ in the previous lemmas.

\begin{lemma} \label{sh} Let
$w, u \in AP_{2n}$ be such that  $w  b = a u $, with $a,b$ paths in  $Q$.
Then $a \in \P$  if and only if $b \in \P$.
\end{lemma}
\begin{proof}
Let  $w=w(p_1, \dots, p_{2n-1})$, $u=u^{op}(q^1, \dots, q^{2n-1})$
and suppose that  $a \in \P$. To show that  $b \in \P$ it is enough to verify that $s(q^{2n-1}) < t(p_{2n-1})=s(b)$. From the proof of   Lemma \ref{A2}(i) we have that $$s(q^{2n-3}) <
t(p_{2n-3}) \quad \mbox{ and } \quad t(p_{2n-2}) \leq t(q^{2n-3}),$$
and the construction of $w$ imply that $t(p_{2n-3}) \leq s(p_{2n-1}) <
t(p_{2n-2})$.  These inequalities imply that  
$$s(q^{2n-3}) < s(p_{2n-1}) \leq t(q^{2n-3})$$ and therefore,
$t(q^{2n-3}) < t(p_{2n-1})$. Now  the maximality of  $t(q^{2n-3})$
implies that  $s(q^{2n-1}) < t(p_{2n-1})$. 
Similarly one can show
that if $b \in \P$ then $a \in \P$.
\end{proof}

 Thus, item  (ii) of  Lemmas  \ref{A2} and  \ref{A} can
 be written as follows:

\begin{lemma} \label{A3}
Let  $w, u  \in AP_{2n}$,  such that $w  b = a u $ with  $a$ or  $ b$
in $ \mathcal{P}$.Then:
\begin{itemize}
\item [(i)]   There exists $z \in AP_{2n+1}$ such that $z$ divides the path $au$ and $t(z)=t(u)$.
\item[(ii)]   There exists $z \in AP_{2n+1}$ such that $z$ divides the path $wb$ and $s(z)=s(w)$.
\end{itemize}

\end{lemma}

\begin{lemma}  \label{cero en i} Let  $w \in AP_n$ and $\Sub(w) = \{
\zeta_1,\dots ,\zeta_m\}$ with  $w = L_m \zeta_m$. If
$\psi \in \Sub(\zeta_m)$ is such that   $\zeta_m = L(\psi) \psi$,
then $L_mL(\psi)  \in I$.\end{lemma}

\begin{proof} From \cite{B}*{Lemma 3.1} we know that  $AP_{n} = AP^{op}_{n}$, thus
 $$w = w(p_1, \dots ,p_{n-1}) =
w^{op}(q^1,\dots ,q^{n-1})$$ and then $ \zeta_m  =   \zeta_m^{op}(q^2, \dots ,q^{n-1})$ and  $ \psi= \psi^{op}(q^3, \dots ,q^{n-1}).$ So $
t(q_1) \leq s(q^3) = s(\psi)$ and therefore  $q^1$ divides 
$L_mL(\psi)$.
  \end{proof}

\begin{lemma} \label{iguales} Let $\{ \zeta_1, \dots ,\zeta_m \}$ be the ordered set
of all  the  concatenations in  $AP_{2n-1}$ contained in a path
$T$ and satisfying   $T = a_i\zeta_ib_i$, with  $a_i, b_i \in
\mathcal{P}$, $s(\zeta_i) < s(\zeta_{i+1})$. For each $i$, let
  $\Sub(\zeta_i) = \{\psi_1^i,\psi_2^i\} \subseteq AP_{2n-2}$, $\zeta_i =
\psi_1^iR(\psi_1^i) = L(\psi_2^i)\psi_2^i$. Then
  $\psi_2^i = \psi_1^{i+1}$ for all $i = 1, \dots,  m-1$.
\end{lemma}
  \begin{proof} We consider  $\zeta_i , \mbox{ } \zeta_{i+1}$.
The situation can be pictured as follows:
 
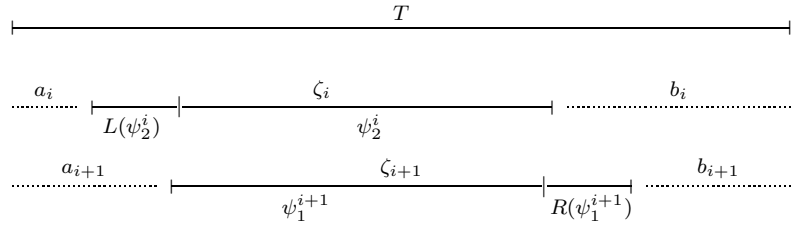
\[  \xymatrix{
\ar@{|-|}[rrrrrrrrrr]^{T} & & & & & & & & & & \\
\ar@{..}[r]^{a_i} &
\ar@{|-|}[rrrrrr]^{\zeta_i}|(.2){|}_(.1){L(\psi_2^i)}_(.6){\psi_2^i} & & & & & & \ar@{..}[rrr]^{b_i}  & & & \\
\ar@{..}[rr]^{a_{i+1}} & &
\ar@{|-|}[rrrrrr]^{\zeta_{i+1}}|(.8){|}_(.3){\psi_1^{i+1}}_(.9){R({\psi_1^{i+1}})}
& & & & & & \ar@{..}[rr]^{b_{i+1}}  & &
 } \]
If  $s(\psi_1^{i+1}) < s(\psi_2^i)$ we should have
\begin{equation*} s(\zeta_{i+1}) =
s(\psi_1^{i+1}) < s(\psi_2^i) < t(\psi_2^i)  = t(\zeta_i) <
t(\zeta_{i+1}). \end{equation*} Then
  $\psi_2^i \in \Sub(\zeta_{i+1})$ with
  $\psi_2^i \neq \psi_1^{i+1}$, $\psi_2^i \neq
  \psi_2^{i+1}$. This is a contradiction since  $  \Sub(\zeta_{i+1}) =  \{\psi_1^{i+1},\psi_2^{i+1}\}$, see \cite{B}*{Lemma 3.3}.

If $s(\psi_2^i) < s(\psi_1^{i+1})$, let $ \delta =
  \psi_2^ib = a\psi_1^{i+1}$. Then  $a $ and $ b$ belong to $ \mathcal{P}$ since  $a$ divides   $a_{i+1}$ and  $b$ divides 
$b_i$. From  Lemma \ref{A3}(i) we deduce that there exists $z \in AP_{2n-1}$ such that
$z$ divides   $\delta$ and  satisfies
$$s(\zeta_i) < s(\psi_2^i) \leq s(z) < t(z) =
t(\psi_1^{i+1}) < t(\zeta_{i+1})$$ Then  $z \not \in
\{ \zeta_1, \dots ,\zeta_m \}$,  a  contradiction.
\end{proof}

Now we will show that  $ {\mathbf F} : \Ap A   \to \Bar A$ is a comparison
morphism. 
It is clear that $ \varepsilon \circ F_0 =  \mu \circ \id_A$.  For $n \geq 1$ we will show that $F_{n-1} \circ d_n = b_n \circ F_n$ 
 inductively. 
If $n = 1$, 
\begin{align*} b_1\circ F_1(1 \otimes \alpha \otimes 1) & =   \alpha \otimes 1 - 1 \otimes \alpha  \\ & =   F_0(\alpha \otimes
e_{t(\alpha)} \otimes 1  - 1\otimes e_{s(\alpha)}\otimes \alpha )
 \\ &  =   F_0\circ d_1 (1 \otimes \alpha\otimes 1). \end{align*}
If  $n = 2$ and   $w =
\alpha_1\dots \alpha_s \in \R$, then 
\begin{align*}  
b_2\circ F_2(1 \otimes w\otimes 1)  =  &  
b_2(\sum_{i= 1}^{s-1} 1\otimes \alpha_1 \dots \alpha_i\otimes \alpha_{i+1}\otimes
\alpha_{i+2} \dots \alpha_s) \\
 = & \sum_{i= 1}^{s-1}
\alpha_1 \dots \alpha_i\otimes \alpha_{i+1}\otimes
\alpha_{i+2} \dots \alpha_s \\
& -   \sum_{i= 1}^{s-1} 1\otimes \alpha_1 \dots \alpha_{i+1}\otimes
\alpha_{i+2} \dots \alpha_s \\
& +  \sum_{i= 1}^{s-1} 1\otimes \alpha_1 \dots \alpha_i\otimes \alpha_{i+1}
\alpha_{i+2} \dots \alpha_s \\  = &
 \sum_{i= 1}^s \alpha_1 \dots \alpha_{i-1}\otimes \alpha_i\otimes
\alpha_{i+1} \dots \alpha_s 
\\ = & F_1(\sum_{i=1}^s
\alpha_1 \dots \alpha_{i-1}\otimes \alpha_i \otimes
\alpha_{i+1} \dots \alpha_s)  \\  =  & F_1\circ d_2(1\otimes w \otimes
1).  \end{align*}
If  $n\geq 3$, we will first show that 
$b_n\circ F_n = F_{n-1}\circ d_n$ 
 for  $n$ odd.  Let $\Sub(w)= \{ \psi_1, \psi_2\}$, $s(\psi_1) < s(\psi_2)$.  Then  
\begin{align*} b_n\circ F_n(1\otimes w\otimes 1)  & =   b_n(1\otimes
L(\psi_2)F_{n-1}(1\otimes \psi_2\otimes 1) ) \\ & = 
L(\psi_2)F_{n-1}(1\otimes \psi_2 \otimes 1) -   1\otimes
L(\psi_2)b_{n-1}F_{n-1}(1\otimes \psi_2\otimes 1)
\\ & =  L(\psi_2)F_{n-1}(1\otimes \psi_2\otimes 1) - 1\otimes
L(\psi_2)F_{n-2}d_{n-1}(1\otimes \psi_2\otimes 1),
\end{align*}
where the first and the second equalities  follow from Remark 
\ref{obs1}(2-3) and  the last equality follows by the induction hypothesis.

On the other hand, 
\begin{align*} F_{n-1}d_n(1\otimes w\otimes 1) & =   F_{n-1}(L(\psi_2)\otimes
\psi_2\otimes 1) - F_{n-1}(1 \otimes \psi_1\otimes R(\psi_1))  \\ &
=  L(\psi_2)F_{n-1}(1\otimes \psi_2\otimes 1) - F_{n-1}(1 \otimes
\psi_1\otimes 1)R(\psi_1). \end{align*}
Then we only have to prove that 
\begin{align*} F_{n-1}(1 \otimes \psi_1\otimes 1)R(\psi_1) = 1\otimes
L(\psi_2)F_{n-2}d_{n-1}(1\otimes \psi_2\otimes 1).
\end{align*}
In fact 
\[F_{n-1}(1 \otimes \psi_1\otimes 1)R(\psi_1) =
\sum_{\substack{\gamma \in \Sub(\psi_1) \\  \vert L_{\gamma } \vert  > 0}}   1\otimes L_{\gamma}
F_{n-2}(1\otimes \gamma \otimes 1) R_{\gamma}R(\psi_1)\] and 
\begin{align*} 1\otimes L(\psi_2)F_{n-2}d_{n-1}(1\otimes
\psi_2\otimes 1) &  =   1 \otimes L(\psi_2)F_{n-2} \left(
\sum_{\gamma \in \Sub(\psi_2)} L(\gamma)\otimes \gamma \otimes
R(\gamma) \right) \\ & =    \sum_{\gamma \in \Sub(\psi_2)}
1\otimes L(\psi_2)L(\gamma)F_{n-2}(1\otimes \gamma\otimes
1)R(\gamma).\end{align*}
If  $\gamma \in \Sub(\psi_1)\cap \Sub(\psi_2)$ the equality of the corresponding summands is clear.
If $\gamma\in \Sub(\psi_1)$ and  $\gamma \not \in \Sub(\psi_2)$,  
then  $s(\gamma) < s(\psi_2)$ and, from Proposition \ref{cero
2}(ii), we have that $R_{\gamma}R(\psi_1) \in I$.
If  $\gamma \in \Sub(\psi_2)$ and  $\gamma \not \in \Sub(\psi_1)$ then  $t(\psi_1) < t(\gamma)$, and from Proposition \ref{cero
2}(i) we have that $L(\psi_2)L(\gamma) \in I $. This finishes the proof for $n$ odd.

If  $n$ is even 
\begin{align*} b_n\circ & F_n(1\otimes w\otimes 1)  =   b_n( \sum _{i= 1}^{m-1}
1\otimes L_{i+1}F_{n-1}(1\otimes \zeta_{i+1}\otimes
1)R_{i+1})  \\ & =  \sum _{i= 1}^{m-1}
 L_{i+1}F_{n-1}(1\otimes
\zeta_{i+1}\otimes 1)R_{i+1}    -   \sum _{i=
1}^{m-1}
 1\otimes L_{i+1}b_{n-1}F_{n-1}(1\otimes
\zeta_{i+1}\otimes 1)R_{i+1}  \\
& = 
  \sum _{i= 1}^{m-1}
 L_{i+1}F_{n-1}(1\otimes
\zeta_{i+1}\otimes 1)R_{i+1}  -  \sum _{i= 1}^{m-1}
 1\otimes L_{i+1}F_{n-2}d_{n-1}(1\otimes
\zeta_{i+1}\otimes 1)R_{i+1}
 \end{align*}
where the second equality  follows by Remark 
\ref{obs1}(3) and  the third  equality follows by the induction hypothesis.
Since   $\zeta_{i+1} \in AP_{n-1}$ and  $n-1$ is odd, $\Sub(\zeta_{i+1})
= \{\psi_1^{i+1}, \psi_2^{i+1} \}$. Then, the above sum is equal to 
\begin{align*}   \sum _{i= 1}^{m-1} 
 L_{i+1}F_{n-1}(1\otimes
\zeta_{i+1}\otimes 1)R_{i+1}  - \sum _{i= 1}^{m-1}
 1\otimes L_{i+1}L(\psi_2^{i+1})F_{n-2}(1\otimes
\psi_2^{i+1}\otimes 1)R_{i+1} \\
+  \sum _{i= 1}^{m-1}
 1\otimes L_{i+1}F_{n-2}(1\otimes \psi_1^{i+1}\otimes 1)
 R(\psi_1^{i+1})R_{i+1}.\end{align*}
From Lemma \ref{iguales} we have that $\psi_2^i = \psi_1^{i+1}$,
for $i = 1, \dots, m-1$. After cancelling the corresponding terms in the above sum, we obtain that it is equal to
\begin{align*}    \sum _{i= 1}^{m-1}
 L_{i+1}F_{n-1}(1\otimes
\zeta_{i+1}\otimes 1)R_{i+1}   -  \ & 1\otimes L_{m}L(\psi_2^m)F_{n-2}(1\otimes
\psi_2^m\otimes 1) \\
+  &  1 \otimes
L_2F_{n-2}(1\otimes \psi_1^2 \otimes 1) R(\psi_1^2)R_{2}.\end{align*}
From Lemma \ref{cero en i}  we get that $
L_{m}L(\psi_2^m) \in I$, and from 
Lemma \ref{iguales}  we have that $\psi_1^2 = \psi_2^1$. This implies that 
 $R(\psi_1^2)R_2 = R_{1}$. Since   $n-1$ is 
odd and   $\Sub(\zeta_1) = \{\psi_1^1, \psi_2^1\}$, 
Remark  \ref{obs1}(2) implies that 
$$F_{n-1}(1\otimes\zeta_1\otimes1) R_1 =  1\otimes L_2F_{n-2}(1\otimes \psi_1^2 \otimes
  1)R(\psi_1^2)R_2.$$
Finally, the  {sum} we are interested in is equal to 
\begin{equation*}     \sum _{i= 1}^{m-1}
L_{i+1}F_{n-1}(1\otimes
\zeta_{i+1}\otimes 1)R_{i+1} +
F_{n-1}(1\otimes\zeta_1\otimes1) R_1  =  F_{n-1}\circ d_n(1\otimes
w\otimes 1)\end{equation*}
and in this way we have completed the proof for  $b_n\circ F_n =
F_{n-1}\circ d_n$.

\section{The  map $\mathbf G$ is a comparison morphism}\label{morfismoG}

First we will prove some preliminary results that will make the computation of the map $\mathbf G$ easier. 
Let $(v_1, \dots, v_n)$  be a well-concatenated $n$-sequence.  If $v_j v_{j+1} \in I$, there exists $\gamma \in \R$ such that $\gamma$ divides the path 
 $v_j v_{j+1}$ and it satisfies \[ s(v_{j}) \leq
s(\gamma) < t(v_{j}) \quad \mbox{and} \quad  s(v_{j+1}) < t(\gamma) \leq
t(v_{j+1}).\]
From now on, we will call  $\gamma_{j}$ one of these relations (we choose one). Thus for each well-concatenated $n$-sequence such that  $\mathcal{M}_{even} (v_1, \dots, v_n)= \emptyset$ we associate a sequence  $(\gamma_2, \gamma_4, \dots)$:

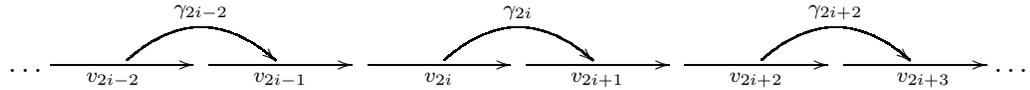
\[ \xymatrix{ \dots \ar[rr]_{v_{2i-2}} & \ar@/^{5mm}/[rr]^{\gamma_{2i-2}} &   \ar[rr]_{v_{2i-1}} & &
\ar[rr]_{v_{2i}} & \ar@/^{5mm}/[rr]^{\gamma_{2i}} &
\ar[rr]_{v_{2i+1}} & & \ar[rr]_{v_{2i+2}} &
\ar@/^{5mm}/[rr]^{\gamma_{2i+2}} & \ar[rr]_{v_{2i+3}} & & \dots
 } \]
 and if  $\mathcal{M}_{odd}  (v_1, \dots, v_n) = \emptyset$,  we associate a sequence $(\gamma_1, \gamma_3, \dots)$ constructed in the same way.

Suppose that there exists $w \in AP_{n-1}$ such that 
 $v_1 \dots v_n = awb  $ with  $a, b$
paths in $\P$. We have the following question:
\begin{center} Does there exist
$z \in AP_n$ such that $w \in \Sub(z)$ and $z \in \chi(v_1, \dots, v_n) $
? 
\end{center}
In the following lemmas we get conditions that ensure a positive answer to the previous question.
\begin{lemma} \label{previo} Let  $(v_1,\dots,v_n) \in  \P^n$ be a  well-concatenated $n$-sequence  and  $w \in AP_{n-1}$   such that 
$v_1\dots v_n = awb$.
\begin{itemize} 
\item[(i)] If  $w = w(p_1, \dots, p_{n-2} )$, $ s(w) < t(v_1)$,  $\mathcal{M}_{even} = \emptyset$  and $(\gamma_2, \gamma_4, \dots )$ is a sequence associated to
$(v_1,\dots ,v_n)$, then $t(p_{n-2}) \leq t(\gamma_{n-2})$ if $n$ is even and $t(p_{n-3}) \leq t(\gamma_{n-3})$ if $n$ is odd.
 \item[(ii)] If $n$ is odd, $w = w^{op}(q^1, \dots, q^{n-2}) $, $ s(v_n) < t(w) $, $\mathcal{M}_{odd} = \emptyset $ and $(\gamma_1, \gamma_3, \dots )$ is a sequence associated to $(v_1,\dots,v_n)$, then $s(\gamma_3) \leq s(q^2)$.    
\end{itemize}
\end{lemma}

\begin{proof}  
(i)  For  $i \geq 1$, we will prove  the  inequalities  $t(p_{2i}) \leq t(\gamma_{2i})$.  
Since   $ s(w) = s(p_1) < t(v_1) = s(v_2) \leq s(\gamma_2)$ then  $s(p_1)
< s(\gamma_2), $ and therefore $t(p_1) < t(\gamma_2).$ By construction of $p_2$ we have that 
$p_2 = \gamma_2 \mbox{ or  } s(p_2) < s(\gamma_2)$, so
$t(p_2) \leq t(\gamma_2).$

Suppose we have already proved that $t(p_{2i-2})  \leq   t(\gamma_{2i-2})$. Since 
$t(\gamma_{2i-2}) < s(\gamma_{2i}) $ we have that   $t(p_{2i-2}) \leq
t(\gamma_{2i-2}) < s(\gamma_{2i})$. By construction of $p_{2i}$ we deduce that  $p_{2i} = \gamma_{2i}$ or $s(p_{2i}) <
s(\gamma_{2i})$, and therefore, $t(p_{2i}) \leq t(\gamma_{2i}). $

(ii) The proof is analogous to $(i)$.
 \end{proof}

From now on let $T$ be the path $v_1 \dots v_n$. 
\begin{lemma}\label{impar} Let  $ n$ odd,  $(v_1,\dots ,v_n) \in  \P^n$ a well-concatenated $n$-sequence, and  $w\in
AP_{n-1}$ such that $v_1\dots v_n = c wb$. Then
\begin{itemize}
\item[(i)] If   $b  \in \P$, $ s(w) < t(v_1)$  and  $\mathcal{M}_{even} = \emptyset$, then there exists $z \in AP_n$ with $z = wR(w)$ such that $R(w)$ divides the path $b$.
\item[(ii)] If $c \in \P$, $s(v_n) < t(w)$ and $\mathcal{M}_{odd} =\emptyset$,  then there exists $z \in AP_n$,  with $z = L(w)w$ such that $L(w)$
 divides the path $c$.
\end{itemize}
\end{lemma}

\begin{proof}
(i) If  $w = w(p_1, \dots ,p_{n-2})$, we have to show that there exists $\delta$ such that 
 $z= z(p_1,\dots ,p_{n-2}, \delta) \in AP_n$.
From Lemma \ref{previo}(i) we know that $t(p_{n-3}) \leq
t(\gamma_{n-3})$ and since
 $b \in \P$, we have that $s(\gamma_{n-1}) <
t(w) = t(p_{n-2})$.

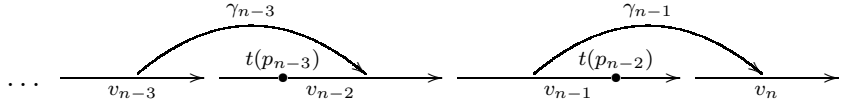
\[ \dots \xymatrix{ \ar[rr]_{v_{n-3}} & \ar@/^{7mm}/[rrr]^{\gamma_{n-3}} &
\ar[rrr]_{v_{n-2}}|(.3){\bullet}^(.3){t(p_{n-3})} & &  &
\ar[rrr]_{v_{n-1}} |(.7){\bullet}^(.7){t(p_{n-2})}&
\ar@/^{7mm}/[rrr]^{\gamma_{n-1}} & & \ar[rr]_{v_{n}} & &   }
\]
Thus the set  $L = \{\gamma \in \R(T): t(p_{n-3})\leq s(\gamma)
< t(p_{n-2})\}$ is not empty, since  $\gamma_{n-1} \in L$. If 
$\delta \in L$ is such that  $s(\delta)$ is minimal with respect to all $\gamma \in L$, we have that  $z =
z(p_1,\dots,p_{n-2}, \delta) \in AP_n$ satisfies the desired conditions. 

\medskip

(ii) The proof is analogous to the previous one; in this case we use  Lemma \ref{previo}(ii). 
 \end{proof}

\begin{lemma}\label{par1} Let  $ n$ even and   let $(v_1,\dots,v_n) \in  \P^n$ be a well-concatenated $n$-sequence. If 
$\mathcal{M}_{even}= \emptyset$  and there exist  $w \in
AP_{n-1}$ such that  $ s(v_1) \leq s(w) < t(v_1)$, then  $t(w) \leq
t(\gamma_{n-2}) \leq t(v_{n-1})$.
\end{lemma}

\begin{proof} If    $w = w(p_1,\dots,p_{n-2})$, from Lemma \ref{previo}(i)
we have that  $t(p_{n-2}) \leq t(\gamma_{n-2})$. Hence  $t(w) =
t(p_{n-2}) \leq t(\gamma_{n-2}) \leq t(v_{n-1}).$
\end{proof}

\begin{prop}\label{par}  Let $n$ even and  let  $(v_1,\dots,v_n) \in  \P^n$ be a well-concatenated $n$-sequence. If
$\mathcal{M}_{odd}= \emptyset $ and there exists  $w =
w(p_1,\dots,p_{n-2}) \in AP_{n-1}$ with  $ v_1\dots v_n = awb $, $a, b \in \P$, then there exists $z \in AP_n$ such that  $z$
divides the path $v_1\dots v_n$ and  $w \in \Sub(z)$.
\end{prop}

We need  some preliminary lemmas to prove this last result.
First, we define a set of relations  $\{
p'_{2i-1}, p''_{2i-1}\}_{1 \leq 2i-1 \leq n-3}$ satisfying:
\begin{itemize}
\item [1)]  $p'_{n-1} = \gamma_{n-1}$ and $p_{2i-1}'$ is an element in the set
$ T_{2i-1} = \{\gamma \in \R(T): t(\gamma) \leq
s(p'_{2i+1})\}$ such that    $t(p_{2i-1}')$ is maximal with respect to all
$\gamma \in T_{2i-1}$;
\item [2)]
$ p_{2i-1}''$ is an element in the set
$ S_{2i-1} = \{\gamma\in
\R(T): s(p_{2i-1}') < s(\gamma)\}$ such that  $s(p''_{2i-1})$
 is minimal with respect to all $\gamma \in S_{2i-1}.$
\end{itemize}
In addition to the assumptions in Proposition \ref{par}, we assume that  $s(\gamma_{n-1}) <
t(p_{n-3})$. Under these conditions we can prove the following  two lemmas. 

\begin{lemma} \label{const1}The sets $T_{2i-1}$ are not empty, more precisely, $\gamma_{2i-1} \in T_{2i-1}$, and the relations $p'_{2i-1}$ satisfy
\begin{equation}
s(\gamma_{2i-1}) \leq s(p'_{2i-1}) < t(p_{2i-3}) \label{eI}
\end{equation} for  $2i-1=1,  \dots, n-3$, where $t(p_{-1}):= s(p_1)$.
\end{lemma}

\begin{proof}
By construction  $t(\gamma_{n-3})
< s(\gamma_{n-1})$ and, since  $p'_{n-1}=\gamma_{n-1}$  we have that  if  $2i-1=n-3$ then
$\gamma_{n-3} \in T_{n-3}$  and hence  $p'_{n-3}$ exists. So, the  maximality of  $t(p'_{n-3})$,  the fact that
$\gamma_{n-3} \in T_{n-3}$  and the assumption $s(\gamma_{n-1}) =
s(p_{n-1}' ) < t(p_{n-3})$, imply that
\begin{equation*} t(\gamma_{n-3}) \leq t(p_{n-3}') \leq s(p_{n-1}') < t(p_{n-3})
\end{equation*}
and  hence  \begin{equation}s(\gamma_{n-3}) \leq s(p_{n-3}') < s(p_{n-3}). \label{contrad} \end{equation}
From the construction of $w$  we know that  $p_{n-3}$ is such that $s(p_{n-3})$ is minimal with respect to all the relations in the set  $
\{ \gamma  \in \R(T):  t(p_{n-5}) \leq s(\gamma) < t(p_{n-4}) \} $.
So, from \eqref{contrad} we get that $p'_{n-3}$ does not belong to the previous set and hence $s(p'_{n-3}) < t(p_{n-5})$. Then $s(\gamma_{n-3}) \leq s(p_{n-3}')  <
t(p_{n-5})$. 

By induction hypothesis, suppose there exists
$p'_{n-2i-1}$ verifying  \eqref{eI}. Now we shall find  $p_{n-2i-3}'$.
From  \eqref{eI},   $s(\gamma_{n-2i -1})   \leq s(p_{n-2i-1}')$ and  since
$t(\gamma_{n-2i-3}) < s(\gamma_{n-2i-1}) $,  $\gamma_{n-2i-3} \in
T_{n-2i-3}$ and this shows the existence of  $p_{n-2i-3}'$. From the  the maximality of $t(p_{n-2i-3}')$ in 
$T_{n-2i-3}$ we have that  $t(\gamma_{n-2i-3}) \leq t(p_{n-2i-3}')$ and 
with  \eqref{eI} we get that 
\begin{equation*}t(\gamma_{n-2i-3}) \leq t(p_{n-2i-3}') \leq
s(p'_{n-2i-1}) < t(p_{n-2i-3}) \end{equation*} therefore
\begin{equation}s(\gamma_{n-2i-3}) \leq s(p_{n-2i-3}')  <
s(p_{n-2i-3})\label{d8}.\end{equation} From the construction of $w$  we know that $p_{n-2i-3}$ is such that 
$s(p_{n-2i-3})$  is minimal with respect to all the relations in $\{\gamma \in \R(T)
: t(p_{n-2i-5}) \leq s(\gamma) < t(p_{n-2i-4}) \}$, where $t(p_{n-2i-5})$ should be replaced by $s(p_1)$ when $n-2i-4=1$.  Then from \eqref{d8} we get that $s(p_{n-2i-3}') < t(p_{n-2i-5})$ and hence 
$s(\gamma_{n-2i-3}) \leq s(p_{n-2i-3}') < t(p_{n-2i-5}).$
\end{proof}

\begin{lemma} \label{const2}The sets $S_{2i-1}$  are not empty, more precisely, $p_{2i-1} \in S_{2i-1}$,  and the relations $p''_{2i-1}$ satisfy
\begin{align}
 s(p_{2i+1}')   & <   t(p_{2i-1}'')   \leq   t(p_{2i-1}), \label{eII} \\
 t(p_{2i-1}') & \leq s(p_{2i+1}')   <
t(p_{ 2i-1}''),  \label{eIII}
\end{align}
for $2i-1= 1 , \dots, n-3$.
\end{lemma}

\begin{proof}
 It is clear that $S_{2i-1} \neq
\emptyset$ because from   \eqref{eI}, $s(p_{2i-1}') < t(p_{2i-3})$ and by
construction of  $w$, $t(p_{2i-3}) \leq s(p_{2i-1})$, so
  $p_{2i-1} \in S_{2i-1}$. From the minimality of $s(p_{2i-1}'')$ in $S_{2i-1}$ we have that 
$s(p_{2i-1}'') \leq s(p_{2i-1})$  and therefore  $t(p_{2i-1}'')
\leq t(p_{2i-1}).$ Also $s(p'_{2i+1}) < t(p''_{2i-1})$,
because if  $t(p''_{2i-1}) \leq s(p'_{2i+1})$ then the maximality of $t(p'_{2i-1})$ in  $T_{2i-1}$ says  that $t(p''_{2i-1})
\leq t(p'_{2i-1})$   and therefore $s(p''_{2i-1}) \leq s(p'_{2i-1})$, and
this contradicts  the definition of  $ p''_{2i-1}$. Then we have that 
$ s(p_{2i+1}') < t(p_{2i-1}'') \leq t(p_{2i-1}). $ Finally, from the definition of  $p_{2i-1}'$ and   the previous inequality 
 we get that 
$ t(p_{2i-1}')\leq s(p'_{2i+1}) < t(p_{2i-1}'').$
\end{proof}

\begin{remark}\label{z2} The relations $p_1', p_1''$ satisfy
$w(p_1',p_1'') \in AP_3$.
\end{remark}

\begin{proof}[Proof  of Proposition \ref{par}]
We assume first that
$t(p_{n-3}) \leq   s(\gamma_{n-1})$. By hypothesis  $v_1 \dots v_n = awb  $ with  $b \in \P$
then  $s(\gamma_{n-1})< t(p_{n-2})$
\[ \xymatrix{  \ar@<-3ex>[rrr]^{p_{n-3}} &   & & &  \ar@<-1ex>[rrr]^{\gamma_{n-1}} &  & & &  \\
& \ar@{-|}[rrrrrrr]^(0.3){p_{n-2}}|(0.6){>}^(0.8)b  & & & & & & & &
. }
\] So, the set $L =       \{\gamma \in \R(T): t(p_{n-3})\leq s(\gamma) <
 t(p_{n-2})\}$ is not empty because  $\gamma_{n-1} \in L$. If $\delta$ is such that $s(\delta)$ is minimal with respect to all the  relations in the set $L$,
we have  that $w = w(p_1,\dots , p_{n-2}, \delta)  \in AP_n$
satisfies the desired  condition.

Assume now that $s(\gamma_{n-1}) < t(p_{n-3}). $ We will construct $z=z(z_1, \dots, z_{n-1}) \in AP_n$ with  $z_1=p'_1$.  From Remark \ref{z2} we have that  $z_2 =
p_1''$ and,  by  (\ref{eII}),  $z_2$ satisfies the inequality
$t(p''_{1})= t(z_{2}) \leq t(p_{1}).$ 

We have to see that the set
$L_i =  \{\gamma \in \R(T) : t(z_{i-2})\leq
s(\gamma) <  t(z_{i-1})\}$  is not empty for $i= 3, \dots, n-1$ and then we will choose $z_i \in L_i$ such that $s(z_i)$ is minimal with respect to all $\gamma \in L_i$. From 
\eqref{eIII} we have that  $p_3' \in L_3$, then $L_3 \neq \emptyset$
and the minimality of  $s(z_3)$ says that  $s(z_3) \leq s(p_3')$, 
and therefore $ t(z_3) \leq t(p_3').$ So we have that 
\begin{equation}t(p_2) \leq t(z_3) \leq t(p_3')  \end{equation}
where the first inequality follows  from the fact that $a \in \P$,  so
$s(p_1) < t(\gamma_1)$, and by  \eqref{eI}, $s(\gamma_1) \leq s(p_1')$
so $t(\gamma_1) \leq t(p_1')$. Then $s(p_1) < t(p_1')
= t(z_1) \leq s(z_3) $ and finally the construction of $p_2$  says that  $
s(p_2) \leq s(z_3)$,  so   $ t(p_2) \leq  t(z_3)$.

By inductive hypothesis we assume that there exist  $z_j$ such that 
\begin{align}   
t(p_{j-1}) & \leq   t(z_{j})   \leq   t(p'_{{j}}) \quad &\mbox{if  $j$ is odd,} \label{d11}\\
t(p''_{j-1})  & \leq  t(z_{j})   \leq   t(p_{j-1})  &\mbox{if  $j$ is even}
\label{d12}
\end{align}
and now we shall find $z_{j+1}$ for  $j+1=2i$ and for  $j+1=2i+1$.  

If  $j+1=2i$, 
the  construction of $w$ says that   $t(p_{j-2})  \leq s(p_{j}) <
t(p_{j-1})$, and by  \eqref{d11} and  \eqref{d12} we have that 
$t(z_{j-1})\leq s(p_{j}) < t(z_{j})$. So $p_{j} \in
L_{j+1}$ and this shows the existence of $z_{j+1}$. The minimality of  $s(z_{j+1})$ says that  $s(z_{j+1}) \leq
s(p_{j})$ and  therefore 
\begin{equation} t(z_{j+1})\leq t(p_{j}).\label{d14}\end{equation}
By   \eqref{eII} and   \eqref{d12} we have that $s(p_{j}')< t(z_{j-1}) \leq
s(z_{j+1})$, then the minimality of  $s(p_{j}'')$ in 
$S_{j}$ says that 
$s(p_{j}'') \leq s(z_{j+1})$ and so  $ t(p_{j}'') \leq t(z_{j+1}).$
 This last inequality with \eqref{d14} says that 
$t(p_{j}'') \leq t(z_{j+1})\leq t(p_{j}) $. 

If $j+1=2i+1$, by   \eqref{eIII},  \eqref{d11} and  \eqref{d12} 
we have that  $t(z_{j-1}) \leq t(p_{j-1}') \leq s(p_{j+1}') <
t(p_{j-1}'') \leq t(z_{j}),$
then   $p_{j+1}' \in L_{j+1} $  and this shows the existence of $z_{j+1}$. By  the minimality of  $s(z_{j+1})$ we have that   $s(z_{j+1}) \leq s(p_{j+1}')$, and  therefore 
$ t(p_{j}) \leq t(z_{j+1}) \leq t(p_{j+1}') $, 
where the first inequality follows  from  the fact that  $s(p_{j})$  is minimal with respect to all relations in
$\{\gamma \in \R(T):  t(p_{j-2}) \leq s(\gamma) <
t(p_{j-1}) \}$  and $z_{j+1}$ belongs to this set.

In this way we have constructed an element
$ z = z(z_1, z_2, \dots, z_{n-1}) \in AP_n$ with $z_{j}$ satisfying \eqref{d11} and  \eqref{d12} respectively.

To finish the proof we must verify that  $w \in \Sub(z)$.
For this, it suffices to show that $s(z) \leq s(w) \mbox{ and  } t(w)
\leq t(z).$ By  \eqref{eI}, $s(z) = s(z_1) = s(p_1') < t(p_{-1}) =
s(p_1) = s(w)$, so we get the first inequality. By 
\eqref{d11}, $t(p_{n-2}) \leq t(z_{n-1}) $ and therefore    $t(w)
\leq t(z).$
\end{proof}
\bigskip

Now, we give some remarks that will make the computation of the map $\mathbf G$ easier to approach.
For any  well-concatenated  $n$-sequence $(v_1, \dots ,v_{n}) \in \P^n$ we define,  if they exist,  \[j_0 = \Min(\mathcal{M}_{odd}), \ j_1
= \Max(\mathcal{M}_{odd})\]  \[i_0 =  \Min(\mathcal{M}_{even}) ,\  i_1 =
\Max(\mathcal{M}_{even}).
\]

\begin{remark}\label{conjuntos} Using the sets $ \mathcal{M}_{odd}, \mathcal{M}_{even}$ we can cancel terms in the sum
\begin{align*} G_{n-1}\circ b_n (1 & \otimes v_1  \otimes \dots \otimes v_n\otimes 1)  
=   G_{n-1}(v_1\otimes v_2\otimes \dots \otimes v_n\otimes 1) \\
& +  \sum_{ \{ j \ : \  2 \ \leq \  2j \  <  \ n \}
}G_{n-1}(1\otimes
v_1\otimes \dots \otimes v_{2j}v_{2j+1}\otimes \dots \otimes v_n\otimes 1) \\
 & -  \sum_{ \{ j \ : \ 1 \ \leq  \ 2j-1 \ <  \ n \}
}G_{n-1}(1\otimes
v_1\otimes \dots \otimes v_{2j-1}v_{2j}\otimes \dots \otimes v_n\otimes 1) \\
 & +  (-1)^nG_{n-1}(1\otimes v_1\otimes \dots \otimes v_n).
\end{align*} 
\begin{itemize}
\item[1)]  If   $ \mathcal{M}_{even} = \emptyset $ and  $\mathcal{M}_{odd} = \emptyset
$ then  
\begin{align*} G_{n-1}\circ b_n (1 \otimes v_1
\otimes \dots \otimes v_n\otimes 1)   =  & G_{n-1}(v_1\otimes
v_2\otimes \dots \otimes v_n\otimes 1) \\ & +    (-1)^nG_{n-1}(1\otimes
v_1\otimes  \dots \otimes v_n),
\end{align*} because   $v_{j}v_{j+1} \in I$ for any  $j$.
\item[2)] If  $\mathcal{M}_{odd} = \emptyset$ and $\mathcal{M}_{even} \neq
\emptyset$ then  \begin{align*}G_{n-1}\circ b_n (1 \otimes & v_1
\otimes \dots \otimes v_n\otimes 1)  =  G_{n-1}(v_1\otimes
v_2\otimes \dots \otimes v_n\otimes 1) \\  & +    \sum_{ \{ j: \
2i_0 \leq 2j \leq 2i_1 \} }G_{n-1}(1\otimes
v_1\otimes \dots \otimes v_{2j}v_{2j+1}\otimes \dots \otimes v_n\otimes 1) \\
&   +  (-1)^nG_{n-1}(1\otimes v_1\otimes \dots \otimes v_n),
\end{align*} because   $v_{2j-1}v_{2j} \in I$ for any  $j$ and  $v_{2j}v_{2j+1} \in
I$ if   $j< i_0$ or  $j> i_1$.
\item[3)] If  $\mathcal{M}_{even} =
\emptyset$ and  $\mathcal{M}_{odd} \neq  \emptyset$ then 
\begin{align*}G_{n-1}\circ b_n (1 \otimes & v_1 \otimes \dots \otimes
v_n\otimes 1) =  G_{n-1}(v_1\otimes v_2\otimes \dots \otimes
v_n\otimes 1)  \\& -  \sum_{ \{ j: \ 2j_0 \leq 2j \leq 2j_1
\} }G_{n-1}(1\otimes
v_1\otimes \dots \otimes v_{2j-1}v_{2j}\otimes \dots \otimes v_n\otimes 1) \\
&  +   (-1)^nG_{n-1}(1\otimes v_1\otimes \dots \otimes v_n),
\end{align*} because   $v_{2j}v_{2j+1} \in I$ for any  $j$ and  $v_{2j-1}v_{2j} \in
I$ if  $j< j_0$ or  $j> j_1$.
\item[4)] If  $\mathcal{M}_{even} \neq
\emptyset$ and  $\mathcal{M}_{odd} \neq  \emptyset$ then 
\begin{align*} G_{n-1}\circ b_n (1 \otimes & v_1 \otimes \dots \otimes
v_n\otimes 1)   =   G_{n-1}(v_1\otimes v_2\otimes \dots \otimes
v_n\otimes 1)  \\ & +  \sum_{ \{ j: \  2i_0 \leq 2j \leq 2i_1
\} }G_{n-1}(1\otimes
v_1\otimes \dots \otimes v_{2j}v_{2j+1}\otimes \dots \otimes v_n\otimes 1) \\
 & -  \sum_{ \{ j: \ 2j_0 \leq 2j \leq 2j_1  \}
}G_{n-1}(1\otimes
v_1\otimes \dots \otimes v_{2j-1}v_{2j}\otimes \dots \otimes v_n\otimes 1)
\\  & +  (-1)^nG_{n-1}(1\otimes v_1\otimes \dots \otimes v_n),
\end{align*} because   $v_{2j}v_{2j+1} \in
I$ if   $j< i_0$ or  $j> i_1$ and  $v_{2j-1}v_{2j} \in I$
if  $j< j_0$ or  $j> j_1$.
\end{itemize}
\end{remark}

The following remark allow us to characterize the kernel of   $G_n$.
\begin{remark}\label{condicioncero} Let  $(v_1, \dots, v_n) \in \P^n$ be a  well-concatenated $n$-sequence.  Then  $$G_n(1 \otimes v_1
\otimes \dots \otimes v_n \otimes 1) = 0$$ if and only if one and
only one of the following conditions is satisfied:
\begin{itemize}
  \item[1)] $(v_1, \dots, v_n)$ is a  bad $n$-sequence,
  \item[2)] $(v_1, \dots, v_n)$ is a  good $n$-sequence  and   $\chi(v_1, \dots, v_n) = \emptyset$,
  \item[3)] $(v_1, \dots, v_n)$ is a  good $n$-sequence, $\chi(v_1, \dots, v_n) \neq
\emptyset$ and 
\begin{itemize}
 \item [i)] if  $n$ is even, $L(w_1) \mbox{ or } R(w_1) \in I$ for  $w_1 \in AP_n$ such that $s(w_1) = \Min\{s(w): w \in \chi(v_1,
\dots, v_n) \}$;
\item [ii)] if  $n$ is odd,  $L(w) \mbox{ or } R(w) \in I$  for any  $w \in \chi(v_1,
\dots, v_n) $ with  $s(v_1)\leq s(w) < t(v_1)$.
\end{itemize}
\end{itemize}
\end{remark}

\begin{remark}\label{conjuntos2}
\begin{itemize}
\item [ ]
  \item[1)]  If  $\mathcal{M}_{odd} \not = \emptyset$ then 
    $\sum_{ \{ j: \ 2j_0 \leq 2j \leq 2j_1 \} }G_{n-1}(1\otimes
v_1\otimes \dots \otimes v_{2j-1}v_{2j}\otimes \dots \otimes v_n\otimes 1)$
is equal to
\begin{align*}
G_{n-1}(1\otimes
v_1\otimes \dots \otimes v_{2j_0-1}v_{2j_0}\otimes \dots \otimes v_n\otimes 1) & \quad \mbox{if $n$ is odd},\\
G_{n-1}(1\otimes v_1\otimes \dots \otimes
v_{2j_1-1}v_{2j_1}\otimes \dots \otimes v_n\otimes 1) &  \quad \mbox{if $n$ is
even}
\end{align*}
because  $(v_1, \dots, v_{2j-1}v_{2j}, \dots, v_n)$ is a bad 
$(n-1)$-sequence if $j>j_0$ when  $n$ is odd since 
$v_{2j_0-1}v_{2j_0} \not \in I$, and if  $j<j_1$ when  $n$ is even 
since  $v_{2j_1-1}v_{2j_1} \not \in I$.  The same argument can be used to affirm that \begin{align*} G_{n-1}(1\otimes
v_1\otimes \dots \otimes v_n) = 0 & \quad  \mbox{if $n$ is odd, and }\\
G_{n-1}( v_1\otimes \dots \otimes v_n\otimes 1) = 0  & \quad \mbox{if $n$ is 
even and  $j_1 \neq 1$. }
\end{align*}
\item[2)]  Analogously if  $\mathcal{M}_{even} \not = \emptyset$ then 
    $\sum_{ \{ j: \ 2i_0 \leq 2j \leq 2i_1 \} }G_{n-1}(1\otimes
v_1\otimes \dots \otimes v_{2j}v_{2j+1}\otimes \dots \otimes v_n)$ is equal to 
\begin{align*}
G_{n-1}(1\otimes
v_1\otimes \dots \otimes v_{2i_0}v_{2i_0+1}\otimes \dots \otimes v_n\otimes 1) & \quad \mbox{if $n$ is even},\\
G_{n-1}(1\otimes
v_1\otimes \dots \otimes v_{2i_1}v_{2i_1+1}\otimes \dots \otimes v_n\otimes 1) & \quad \mbox{if  $n$ is odd}, \end{align*} and 
\begin{align*}G_{n-1}(1\otimes
v_1\otimes \dots \otimes v_n) = 0 & \quad  \mbox{if $n$ is even, }\\
G_{n-1}( v_1\otimes \dots \otimes v_n\otimes 1) = 0 & \quad  \mbox{if $n$ is
odd. }
\end{align*}
    \end{itemize}
\end{remark}

Now we will show that  $ {\mathbf G} : \Bar A   \to \Ap A$ is a comparison
morphism. 
It is clear that $ \mu \circ G_0   =
\varepsilon \circ \id_A$ and  $G_0 \circ b_1 = d_1 \circ G_1$ since 
\begin{equation*}
   d_1 \circ G_1(1\otimes v \otimes 1) = 0 =  G_0 \circ b_1(1\otimes
v \otimes 1) 
\end{equation*}
if $\vert v \vert =0$ and, if $ v = \alpha_1 \dots \alpha_s$,
\begin{align*}
 d_1 \circ G_1(1\otimes & \alpha_1 \dots \alpha_s \otimes 1) =
d_1(\sum_{i=1}^s\alpha_1 \dots \alpha_{i-1}\otimes \alpha_i\otimes
\alpha_{i+1} \dots \alpha_s) \\ 
   = & \sum_{i=1}^s
\alpha_1 \dots \alpha_{i-1} \alpha_i\otimes e_{t(\alpha_i)}\otimes
\alpha_{i+1} \dots \alpha_s \\
& - \sum_{i=1}^s\alpha_1 \dots \alpha_{i-1}\otimes
e_{s(\alpha_i)} \otimes \alpha_i \alpha_{i+1} \dots \alpha_s \\   = &
\alpha_1 \dots \alpha_s\otimes e_{t(\alpha_s)}\otimes 1  -  1 \otimes
e_{s(\alpha_1)}\otimes \alpha_1 \dots \alpha_s  = G_0 \circ b_1(1\otimes
v \otimes 1).
\end{align*}
For  $n > 1$,  the  proof of $d_n\circ G_n = G_{n-1}\circ b_n$
will be done in four   steps  taking into account Remarks \ref{conjuntos} and  \ref{conjuntos2}.

\subsubsection*{\underline {Case 1}:  $ \mathcal{M}_{even} (v_1, \dots, v_n)= \emptyset \mbox{  and } \mathcal{M}_{odd} (v_1, \dots, v_n) = \emptyset $.} In this case
\begin{align*} G_{n-1}\circ b_n (1 \otimes v_1
\otimes \dots \otimes v_n\otimes 1)   = &  G_{n-1}(v_1\otimes
v_2\otimes \dots \otimes v_n\otimes 1) \\ & +    (-1)^nG_{n-1}(1\otimes
v_1\otimes \dots \otimes v_n).
\end{align*}
(i) If $G_n (1 \otimes v_1 \otimes \dots \otimes v_n\otimes 1) =0$  we must show that  \begin{equation}
G_{n-1}(v_1
 \otimes v_2\otimes \dots \otimes v_n \otimes 1) =  -(-1)^n G_{n-1}( 1 \otimes v_1 \otimes  \dots \otimes
v_n). \tag{$\ast$}\end{equation}
Assume  $n$ is odd. First we observe that if  $G_{n-1}( 1 \otimes
v_1 \otimes \dots \otimes v_n) \neq 0$ then
\[  G_{n-1}( 1 \otimes v_1 \otimes \dots \otimes
v_n)= L(w)\otimes w\otimes L(w) ,\] with  $L(w) , R(w) \in \P$ and  $w
\in AP_{n-1}$.  If $ s(v_1) \leq s(w) < t(v_1)$, by  Lemma \ref{impar}(i) there exists $z \in AP_n$ with  $s(z) = s(w)$ and 
therefore  $ s(v_1) \leq s(z) < t(v_1)$ contradicting that
$G_n(1\otimes v_1\otimes \dots \otimes
   v_n\otimes 1) = 0$.
Then $t(v_1)= s(v_2) \leq s(w)$. So, in this case,  equality $(\ast)$ holds   since  $w \in \chi  (v_2, \dots, v_n)$.

Assume now that 
$$G_{n-1}( 1 \otimes v_1 \otimes
\dots \otimes v_n)  = 0   \mbox{ \   and  \ }    G_{n-1}(v_1
 \otimes v_2\otimes \dots \otimes v_n \otimes 1)   = L(w) \otimes w\otimes R(w)\neq 0 $$
with  $w \in \chi(v_2, \dots, v_{n})$ and  $s(w)$ minimal. Since $(v_1, \dots, v_{n-1}) $ is a good  $(n-1)$-sequence, the  first  equality says that we are in case  (2) or (3.i) of 
Remark \ref{condicioncero}.

In the first case, $\chi(v_1, \dots, v_{n-1}) = \emptyset$ and therefore   $ w \in
AP_{n-1}$ satisfies  $s(v_n) < t(w) \leq t(v_n)$. Then 
by Lemma \ref{impar}(ii) there exists $z \in AP_n$, with  $t(z) =
t(w)$ and  $w \in \Sub(z)$. If  $ t(v_1)\leq s(z)$, we have that there exists $w' \in \Sub(z)$ with 
$s(w') = s(z)$, and  this contradicts the minimality of $s(w)$.
If  $s(z) < t( v_1)$ we get a contradiction to  $G_n(1\otimes
v_1\otimes \dots \otimes
   v_n\otimes 1) =  0$.

In the second case  $\chi(v_1, \dots, v_{n-1}) \neq \emptyset$, and if 
$u \in AP_{n-1}$ is such that  $s(u)$  is minimal with respect to the elements in $\chi(v_1, \dots, v_{n-1})$ then $L(u)$ or  $R(u) \in I$. Observe that 
 $u$  and $w$ verify  that  $aw = ub$, with 
$a \in \P$ because  $a$ divides   $L(w)$:
\[ \xymatrix{  \ar@{|-|}[rrrr]^(.4){u}|(.8){|}^(.9){b}   \ar@<-3ex>@{|-|}[rrrr]^(.1){a}|(.2){|}^(.6){w} & & & & }  \]
Then,  by Lemma \ref{A3}(i), there exists  $z \in AP_n$ such that 
$t(z) = t(w)$: If $s(v_1) \leq s(z) < t(v_1) $, then  $G_n(1\otimes
v_1\otimes \dots \otimes v_n\otimes 1) \neq 0$, a contradiction; if  $s( v_2) \leq s(z)$ we have that there exists $w' \in \Sub(z) $ with  $s(w') =
s(z)$, and this contradicts the minimality of $s(w)$.\\

Now assume that $n$ is even. Since $(v_1, \dots, v_n)$ is a good   $n$-sequence,
we are  in case (2) or  (3.i) of Remark \ref{condicioncero}.

If $\chi(v_1, \dots ,v_n) =
\emptyset$
then $G_{n-1}(v_1 \otimes v_2 \otimes \dots
\otimes v_n \otimes 1) = 0$ and  $G_{n-1}(1 \otimes v_1 \otimes \dots
\otimes v_n) = 0$ because otherwise using Proposition \ref{par} we obtain a contradiction.

If  $ \chi(v_1, \dots
,v_n) \neq \emptyset$,  we know that  $L(w)  \in I$ or  $R(w)\in I$, where $w \in AP_n$ is such that  $s(w)$
 is minimal with respect to $\chi (v_1, \dots, v_n)$.
Suppose that    
\[G_{n-1}(v_1 \otimes v_2 \otimes \dots
\otimes v_n \otimes 1) \neq 0\quad \mbox{ or } \quad G_{n-1}(1 \otimes v_1
\otimes \dots \otimes v_n) \neq 0; \] then there must exists $u \in AP_{n-1}$
with  $v_1\dots v_n = L(u)uR(u)$, and  $L(u),R(u) \in \P$ and again 
using Proposition \ref{par}  we construct $z \in AP_n$ with  $u \in
\Sub(z)$. Observe that $w \neq z$ because $L(w)  \in I$ or  $R(w)\in
I$.  Since $s(w)$  is minimal,  then  $s(w) < s(z)$ and since $ L(w)$ divides  $ L(u)$, we have that $L(w) \in \P$. Then $R(w)
\in I$ and so  $t(w) < t(u)$
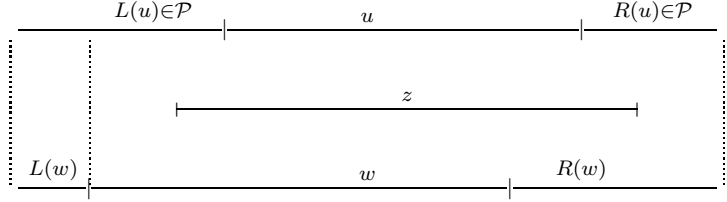
\[ \xymatrix{\ar@{-}[rrrrrrrrr]^{u}^(.2){L(u) \in \P}^(.9){R(u) \in
\P}|(.3){|}|(.8){|}
  \ar@{..}[dd] &  \ar@{..}[dd] & & & & &
& & & \ar@{..} [dd] \\  & & \ar@{|-|}[rrrrrr]^{z} & & & & & & &  \\
\ar@{-}[rrrrrrrrr]^{w}^(.06){L(w)}^(.8){R(w)}|(.11){|}|(.7){|} & & &
& & & & & & }
\]
Now we compare $w$ with $z$ and $w$ with $u$: 
since  $az = wb$  with  $a \in  \P$, because $a$ divides $L(u)
$, by  Lemma \ref{sh}  we have  that   $b \in \P$, 
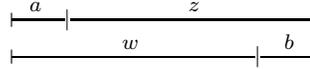
\[ \xymatrix{ \ar@{|-|}[rrrr]^(.1){a}|(.2){|}^(.6){z} \ar@<-3ex>@{|-|}[rrrr]^(.4){w}|(.8){|}^(.9){b} &
& & &    }  \]
so, since $w b'= cu$ we have that   $b'
\in \P$ because   $b'$ divides $ b$.  Finally,  applying  Lemma \ref{A}(i) 
we get that 
$c \in I$,  a contradiction because   $c$ divides  $L(u)$ and $L(u)
\in \P$. \\
(ii)  Assume that  $G_n (1 \otimes v_1 \otimes \dots \otimes v_n\otimes 1) \not = 0$. 
If  $n$ is odd
$$G_n(1\otimes v_1\otimes \dots \otimes
   v_n\otimes 1) = \sum_{i=1}^m  a_i\otimes \zeta_i \otimes b_i \ , \mbox{  }  s(v_1) \leq s(\zeta_i) < t(v_1),$$
where  $\{\zeta_1, \dots ,\zeta_m\} \subseteq
  AP_n$ is an ordered  set with  $s(\zeta_i)<s(\zeta_j)$ if  $i<j$.
If  $\Sub(\zeta_i) = \{
  \psi_1^i,\psi_2^i\}$ by  Lemma \ref{iguales} we have that  $\psi_1^{i+1} = \psi_2^{i}$ for  $i= 1, \dots,  m-1$. Then 
 we can cancel terms in  $d_n\circ G_n$:
\begin{align*} d_n\circ G_n (1\otimes v_1\otimes \dots \otimes
   v_n\otimes 1) &  =   \sum_{i=1}^m a_iL(\psi_2^i)\otimes \psi_2^i \otimes b_i - \sum_{i=1}^m a_i\otimes \psi_1^i \otimes R(\psi_1^i)b_i \\
   & =   a_mL(\psi_2^m) \otimes
   \psi_2^m\otimes  b_m - a_1\otimes \psi_1^1\otimes
  R(\psi_1^1) b_1 .  \end{align*}
We shall prove  that 
\begin{equation}G_{n-1}( v_1 \otimes v_2
 \otimes \dots \otimes v_n \otimes 1)  =    a_mL(\psi_2^m) \otimes
   \psi_2^m\otimes  b_m , \quad \mbox{and}  \label{I} \end{equation}
\begin{equation} G_{n-1}(1 \otimes v_1\otimes \dots \otimes v_n)   =   a_1\otimes
\psi_1^1\otimes
  R(\psi_1^1) b_1 \label{II} . \end{equation}
Suppose first  that $a_mL(\psi_2^m)
\otimes
   \psi_2^m\otimes  b_m \neq 0,$ that is, $a_mL(\psi_2^m), b_m \in \P$. We will see that  $\psi_2^m$  is such that 
    $$s(\psi_2^m) = \Min\{s(w): w \in \chi(v_2, \dots, v_n)\}$$
and this proves equality \eqref{I}. If  $s(v_1) \leq s(\psi_2^m) < t(v_1)$ by  Lemma
\ref{impar}(i) there exists  $z \in AP_n$ with  $s(z) = s(\psi_2^m)$, and this contradicts the maximality of $s(\zeta_m)$, because  $s(\zeta_m) < s(\psi_2^m) = s(z)$. Then 
$t(v_1) \leq s(\psi_2^m)$ and  $\psi_2^m \in \chi(v_2, \dots,
v_n).$ It only remains to see that  $s(\psi_2^m)$ is minimal. If  $\delta \in
\chi(v_2, \dots, v_n) $ satisfies  $s(\delta)< s(\psi_2^m)$
\[ \xymatrix{
\ar@{|-|}[rrrr]^(.4){\delta}|(.8){|}^(.9){b} \ar@<-3ex>@{|-|}[rrrr]^(.1){{a \in \P}}|(.2){|}^(.6){\psi_2^m} & & & & }
\]
 by Lemma \ref{A3}(i), there exists   $z \in AP_n$ with  $t(z) = t(\psi^m_2)$. Hence 
 $z = \zeta_m$ and since $ t(v_1)\leq s(\delta) \leq  s(z)
= s(\zeta_m)$ we obtain a contradiction with $s(\zeta_m) <t(v_1)$. 

On the other hand, suppose that $a_mL(\psi_2^m)
\otimes
   \psi_2^m\otimes  b_m   =   0$ and $G_{n-1}( v_1 \otimes v_2
 \otimes \dots \otimes v_n \otimes 1) = a\otimes\delta \otimes b   \neq
 0$. Then we have two posibilities, that is, 
$s(\psi_2^m) < s(\delta)$ or  $s(\delta) < s(\psi_2^m)$:
\[ \xymatrix{  \ar@{|-|}[rrrr]^{\psi_2^m}
& \ar@<-3ex>@{|-|}[rrrr]^{\delta}  & & &  & \quad \mbox{ or } \quad
  \ar@{|-|}[rrrr]^{\delta} & \ar@<-3ex>@{|-|}[rrrr]^{\psi_2^m}  & & &
&  } \] and again,  by Lemma \ref{A3}(i) there exists 
$z \in AP_{n} $.  In the first case, $z$ satisfies  $t(z) = t(
\delta)$ and this contradicts the maximality of  $s(\zeta_m)$, because 
$s(\zeta_m) < s(\psi_2^m) \leq s(z)< t(v_1)$ (observe that 
if  $t(v_1)\leq s(z)$, taking  $w_1 \in \Sub(z)$ with $z = w_1R(w_1)$ should imply that
$s(\delta)$ is not minimum).  In the second case, $t(z) = t(
\psi_2^m)$ and this contradicts that $s(\zeta_m) < t(v_1)$, because it should be
$\zeta_m = z$ and  $t(v_1) \leq s(\delta) \leq s(z)$.

The proof for equality \eqref{II} is similar to that for \eqref{I}: first assume  that  $a_1\otimes
\psi_1^1\otimes
   R(\psi_1^1)b_1 \neq 0$, then  we use Lemmas  \ref{impar}(ii) and  \ref{A3}(ii)  to see that $s(\psi_1^1) = \Min\{s(w): w \in \chi(v_1, \dots, v_{n-1})\}.$ On the other hand,  if  $a_1\otimes \psi_1^1 \otimes
   R(\psi_1^1)b_1 =  0$   and  $G_{n-1}(1\otimes v_1 \otimes v_2
 \otimes \dots \otimes v_n ) = a \otimes \delta \otimes b  \neq   0$
we have that  $s(\delta) < s(\psi_1^1)$, and  by  Lemma
\ref{A3}(ii)  we obtain a contradiction with the minimality of 
$s(\zeta_1)$.   \\

If  $n$ is even we have that   $G_n(1 \otimes v_1 \otimes \dots
\otimes v_n \otimes 1) = a\otimes w \otimes b$
with  $w \in AP_n$, $s(w)$ minimal and  $a, b \in \P$. Then 
\begin{align*} d_n\circ G_n(1\otimes v_1\otimes \dots \otimes v_n
\otimes 1) & =   d_n(a\otimes w \otimes b)  =  \sum_{w_i \in
\Sub(w) } aL(w_i) \otimes w_i\otimes R(w_i)b.
\end{align*}
Since   $v_1v_2 \in I$, because   $\mathcal{M}_{odd} = \emptyset$,
we only have to consider
$w_i \in \Sub(w)$ with   $s(v_1) \leq s(w_i) < t(v_2) $.
Hence the above sum can be written as   \begin{align*}
\mathcal{J} =
\sum_{\substack{ w_i \in \Sub(w) \\ s(v_1) \leq s(w_i) < t(v_1) }}
aL(w_i) \otimes w_i\otimes R(w_i)b &  +   \sum_{\substack{ w_i \in
\Sub(w) \\ s(v_2) \leq s(w_i) < t(v_2)  }} aL(w_i) \otimes w_i\otimes
R(w_i)b
\end{align*}
and we must show that $\mathcal{J} =  G_{n-1}(
v_1\otimes v_2\otimes \dots \otimes v_n \otimes 1) + G_{n-1}(1 \otimes
v_1\otimes \dots \otimes v_n).$  If  $w_i \in \Sub(w)$ and $ s(v_1) \leq s(w_i) < t(v_1) $ from  Lemma \ref{par1} we deduce that $t(w_i) \leq s(v_n) = t(v_{n-1})$, then the desired equality will follow by the equalities 
\begin{align*}
\sum_{\substack{ w_i \in \Sub(w) \\ s(v_1) \leq s(w_i) < t(v_1) }}
aL(w_i) \otimes w_i\otimes R(w_i)b & =   G_{n-1}(1 \otimes
v_1\otimes \dots \otimes v_n), \quad \mbox{and}\\
  \sum_{\substack{ w_i \in
\Sub(w) \\ s(v_2) \leq s(w_i) < t(v_2)  }} aL(w_i) \otimes w_i\otimes
R(w_i)b &= G_{n-1}(
v_1\otimes v_2\otimes \dots \otimes v_n \otimes 1).
\end{align*}
It is clear that any term in the sums on the left hand corresponds to a
term in the sums on the right hand. Reciprocally, let $\delta \in AP_{n-1}$ with 
 $c \otimes \delta \otimes d$ be a term in the sums on the right hand,   and assume that
 $\delta \not \in \Sub(w)$, then   $s(w)
< s(\delta)$  or $s(\delta) < s(w)$:
\[ \xymatrix{  \ar@{-}[rrrr]^{w}|(.2){|}^(.1){a}|(.65){|}^(.75){b} \ar@<-3ex>@{-}[rrrr]^{\delta} |(.3){|}|(.75){|}  & &  &   & \quad \mbox{ or } \quad
  \ar@{-}[rrrr]^{w}|(.36){|}^(.25){a}|(.8){|}^(.9){b} \ar@<-3ex>@{-}[rrrr]^{\delta} |(.3){|}|(.75){|}  & &  &   & } \]  
In the first case we have that $ w b'=  c'\delta $, with  $b'$ a divisor of $
b$ and therefore $b' \in \P$. By Lemma  \ref{A}(i) we have that 
 $c' \in I $ and then   $c \in I $.  So $c \otimes \delta
\otimes d = 0$.
Similarly in the second case one sees that $d \in I$ using  Lemma
\ref{A2}(i).

\subsubsection*{\underline{Case 2}: $ \mathcal{M}_{even} (v_1, \dots, v_n)  \neq \emptyset $ and $\mathcal{M}_{odd} (v_1, \dots, v_n) = \emptyset $.}

(i) Assume that $G_n (1 \otimes v_1 \otimes \dots \otimes v_n\otimes 1) =0$.
If  $n$ is odd then 
\begin{align*} G_{n-1}\circ b_n(1\otimes v_1 \otimes \dots \otimes v_n \otimes 1)   =  & G_{n-1}(1 \otimes v_1 \otimes \dots \otimes v_{2i_1}v_{2i_1+1}
\otimes \dots \otimes v_n \otimes 1) \\ & -   G_{n-1}(1 \otimes
v_1\otimes \dots \otimes v_n)\end{align*} and hence we must show that 
\begin{equation} G_{n-1}(1 \otimes v_1 \otimes \dots \otimes
v_{2i_1}v_{2i_1+1} \otimes \dots \otimes v_n \otimes 1) = G_{n-1}(1
\otimes v_1\otimes \dots \otimes v_n).\tag{$\ast$}\end{equation} 
Since $\chi(v_1, \dots ,v_{n-1})\subseteq \chi(v_1, \dots ,v_{2i_1}.v_{2i_1+1}, \dots , v_n)$, 
clearly if $\chi(v_1, \dots ,v_{n-1})  \not = \emptyset$
we obtain the equality $(\ast)$.

If $\chi(v_1, \dots ,v_{n-1})   = \emptyset$   and $$G_{n-1}(1 \otimes
v_1 \otimes \dots \otimes v_{2i_1}v_{2i_1+1} \otimes \dots \otimes v_n
\otimes 1) =  L(w)\otimes w \otimes R(w) \neq 0,$$ with   $w \in
AP_{n-1}$ and   $s(w)$ minimal, we should have that   $s(v_n)
< t(w) \leq t(v_n)$. Then by  Lemma \ref{impar}(ii) there would exist  $z \in AP_n$ with  $s(z) < s(w)$. In this case we can take $w' \in \Sub(z)$
with $s(w')= s(z)$ and  this contradicts the minimality of $s(w)$.

If $n$ is even the proof follows exactly as in  the case $\mathcal{M}_{odd} = \mathcal{M}_{even} =
\emptyset$. \\
 (ii) If  $G_n (1 \otimes v_1 \otimes \dots \otimes v_n\otimes 1) \not = 0$, $n$ must be even because  $(v_1, \dots , v_n)$ is a good  $n$-sequence and the proof follows  as in  the case $\mathcal{M}_{odd} = \mathcal{M}_{even} =
\emptyset$ (note that in this case it is not necessary to use Lemma  \ref{par1}).

\subsubsection*{\underline{Case 3}: $ \mathcal{M}_{even} (v_1, \dots, v_n) = \emptyset $ and  $\mathcal{M}_{odd} (v_1, \dots, v_n)  \neq \emptyset
$.}

(i) Assume $G_n (1 \otimes v_1 \otimes \dots \otimes v_n\otimes 1)
=0$.
 If $n$ is odd then 
\begin{align*}G_{n-1}\circ b_n(1\otimes v_1 \otimes \dots \otimes v_n
\otimes 1)  = &  G_{n-1}(v_1
 \otimes v_2\otimes \dots \otimes v_n \otimes 1)\\ &  -  G_{n-1}( 1 \otimes v_1 \otimes  \dots \otimes
v_{2j_0-1}v_{2j_0} \otimes \dots \otimes v_n \otimes 1). \end{align*}
and the proof follows  as in  the case $\mathcal{M}_{odd} = \mathcal{M}_{even}
= \emptyset$.

If $n$ is even then 
\begin{align*} 
G_{n-1}b_n (1 \otimes & v_1 \otimes \dots \otimes v_n\otimes 1) 
=   G_{n-1}( v_1 \otimes v_2 \dots \otimes v_n\otimes 1) \\ 
&  -  G_{n-1} (1 \otimes v_1v_2 \otimes \dots v_n\otimes 1)   + 
G_{n-1} (1\otimes v_1\otimes \dots \otimes v_n) 
\end{align*} 
 if  $j_1 =1$,  and
  \begin{align*}
 G_{n-1}b_n (1 \otimes v_1 \otimes \dots \otimes v_n\otimes 1) &
=  G_{n-1} (1\otimes v_1\otimes \dots \otimes v_n)\\ &  - 
G_{n-1}(1\otimes v_1\otimes \dots \otimes
v_{2j_1-1}v_{2j_1}\otimes \dots \otimes v_n\otimes 1) \end{align*} 
if $j_1 \neq 1 $.  By  Lemma
\ref{par1}, if $w \in AP_{n-1}$ and  $s(w)< t(v_1)$ then  $t(w) \leq
s(v_n) = t(v_{n-1})$. Then using the definition of
$G_{n-1}$ it is clear that  $$G_{n-1} (1 \otimes v_1v_2 \otimes \dots v_n\otimes 1) =   G_{n-1}( v_1 \otimes v_2 \dots \otimes v_n\otimes
1) + G_{n-1} (1\otimes v_1\otimes \dots \otimes v_n)$$ in the first case, and  $$ G_{n-1} (1\otimes v_1\otimes \dots \otimes v_n) =
G_{n-1}(1\otimes v_1\otimes \dots \otimes
v_{2j_1-1}v_{2j_1}\otimes \dots \otimes v_n\otimes 1)$$ in the second  case. \\
(ii)  If $G_n (1 \otimes v_1 \otimes \dots \otimes v_n\otimes 1) \not
=0$, by definition $(v_1, \dots ,
v_n)$ is a good $n$-sequence and therefore   $n$ must be odd. In this case we have that 
\begin{align*} G_{n-1}b_n(1\otimes v_1 \otimes \dots \otimes v_n
\otimes 1)   = &  G_{n-1}( v_1 \otimes v_2
 \otimes \dots \otimes v_n \otimes 1)\\ &  -  G_{n-1}(1 \otimes
 v_1\otimes \dots \otimes v_{2j_0-1}v_{2j_0}\otimes \dots
\otimes v_n \otimes 1).\end{align*}
The proof  follows as in the case 
$\mathcal{M}_{odd} = \mathcal{M}_{even} = \emptyset$.

\subsubsection*{\underline{Case 4}: $ \mathcal{M}_{even} (v_1, \dots, v_n) \neq \emptyset $ and  $\mathcal{M}_{odd} (v_1, \dots, v_n)  \neq \emptyset $.}

In this case the $n$-sequence is bad and so $G_n (1 \otimes v_1 \otimes \dots \otimes v_n\otimes
1) =0$.
If  $n$ is odd then 
\begin{align*} G_{n-1}\circ b_n(1\otimes v_1 \otimes \dots \otimes
v_n \otimes 1)   = & G_{n-1}(1 \otimes v_1 \otimes \dots \otimes
v_{2i_1}v_{2i_1+1} \otimes \dots \otimes v_n \otimes 1) \\ &  - 
G_{n-1}(1 \otimes v_1\otimes \dots \otimes
v_{2j_0-1}v_{2j_0}\otimes \dots \otimes v_n \otimes 1).
\end{align*}
Observe that the $(n-1)$-sequence  $( v_1, \dots , v_{2i_1}v_{2i_1+1},\dots , v_n)$ is 
good  if and only if  the  $(n-1)$-sequence $( v_1,\dots , v_{2j_0-1}v_{2j_0},\dots ,v_n)$ is good, thus
\[\chi(v_1,\dots , v_{2i_1}v_{2i_1+1},\dots , v_n) = \chi(v_1,\dots , v_{2j_0-1}v_{2j_0},\dots  ,v_n)\]
and, by definition,  we get that
\begin{equation*} G_{n-1}(1 \otimes v_1 \otimes \dots \otimes
v_{2i_0}v_{2i_0+1} \otimes \dots \otimes v_n \otimes 1) = G_{n-1}(1
\otimes v_1\otimes \dots \otimes v_{2j_0-1}v_{2j_0}\otimes \dots \otimes v_n
\otimes 1).\end{equation*}

If  $n$ is even we have that 
\begin{align*} G_{n-1}\circ b_n (1 \otimes & v_1 \otimes \dots \otimes v_n\otimes 1) 
=  G_{n-1}( v_1 \otimes v_2 \dots \otimes v_n\otimes 1)\\ &  - 
 G_{n-1} (1 \otimes v_1v_2\otimes \dots  \otimes v_n\otimes 1)  \\
 & + G_{n-1}(1\otimes v_1\otimes \dots \otimes
v_{2i_0}v_{2i_0+1}\otimes \dots \otimes v_n\otimes 1) 
\end{align*}
if $j_1 = 1$ and
\begin{align*}
G_{n-1}\circ b_n (1 \otimes & v_1
\otimes \dots \otimes v_n\otimes 1)  =   G_{n-1}(1\otimes
v_1\otimes \dots \otimes v_{2i_0}v_{2i_0+1}\otimes \dots \otimes v_n\otimes
1) \\ & -  G_{n-1} (1 \otimes v_1\otimes \dots  \otimes
v_{2j_1-1}v_{2j_1}\otimes \dots  v_n\otimes 1),\end{align*}
if $j_1\neq 1$, and again by definition we get that
\begin{align*} G_{n-1} (1 \otimes v_1v_2\otimes \dots \otimes v_n
\otimes 1) = &  G_{n-1}( v_1 \otimes v_2 \dots \otimes v_n\otimes
1)\\ & +  G_{n-1}(1\otimes v_1\otimes \dots \otimes
v_{2i_0}v_{2i_0+1}\otimes \dots \otimes v_n\otimes 1)
\end{align*} in the first case and 
  \begin{equation*} G_{n-1}(1\otimes v_1\otimes \dots \otimes
v_{2i_0}v_{2i_0+1}\otimes \dots \otimes v_n\otimes 1) =  G_{n-1} (1
\otimes v_1\otimes \dots \otimes v_{2j_1-1}v_{2j_1}\otimes \dots
v_n\otimes 1)\end{equation*} in the second case.

\section{The Gerstenhaber algebra $\HH^{\ast}(A)$ for a monomial algebra $A$}

We use the comparison morphisms to obtain formulas that allow us to calculate  the ring and the Lie algebra structure of the Hochschild cohomology of monomial algebras. 

The technique consists in transporting the two structures defined on the complex $\Hom_{A^e}({\Bar A}, A)$ to structures defined on the complex  $\Hom_{A^e}({\Ap A}, A)$, using the quasi-isomorphisms 
$${\mathbf F}^{\bullet}  = \Hom_{A^e}( {\mathbf F}, A)  =   \left (
F^{n} : \Hom_{A^e}(  A^{\otimes^{n+2}}  , A) \longrightarrow \Hom_{A^e}(
A \otimes \k AP_n \otimes A , A)\right )_{n\geq 0}$$
and  
$$ {\mathbf G}^{\bullet} = \Hom_{A^e}( {\mathbf G}, A)
 =  \left ( G^{n} : \Hom_{A^e}(A \otimes  \k AP_n \otimes A , A)
 \longrightarrow \Hom_{A^e}( A^{\otimes^{n+2}} , A)\right )_{n\geq 0}$$
 induced  by the morphisms ${\mathbf F}$ and ${\mathbf G}$ respectively. This is done as follows (we still denote $\cup$ and $[ \ , \ ]$ the products defined using Bardzell's complex):
given $f \in \Hom_{A^e}(A \otimes \k AP_n \otimes A , A)$ and $g \in \Hom_{A^e}(A\otimes \k AP_m \otimes A, A)$,
$$f \cup g \in  \Hom_{A^e}(A \otimes \k AP_{n+ m} \otimes A, A) \quad \mbox{and}$$
$$ f \circ_i g \in \Hom_{A^e}(A\otimes \k AP_{n+ m-1 } \otimes A , A), \quad  1 \leq i \leq n $$
are defined by 
$$f\cup g  =  F^{n+m}(  G^n(f) \cup  G^m(g)) \qquad \mbox{and} \qquad  f\circ_i g    =    F^{m+n-1}(  G^n(f) \circ_i G^m(g)).$$

The same idea has been used in \cite{RR}*{Section 4} in order to describe the Gerstenhaber structure of the Hochschild cohomology of a quadratic string algebra.

Now we present an example to show this technique and later we conclude with two applications.
\begin{example} Let $A = \k Q/I$  where     the quiver $Q$ is an oriented cycle of length $r$ 
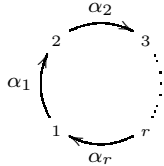
$$\xymatrix{ \scriptscriptstyle2\ar@/^/[r]^{\alpha_2} & \scriptscriptstyle3 \ar@{.}@/^/[d] &  \\ \scriptscriptstyle1  \ar@/^/[u]^{\alpha_1}  & \scriptscriptstyle r \ar@/^/[l] ^{\alpha_r} 
   }$$ and  $I = <\alpha_{1}\alpha_{2}\dots\alpha_r\alpha_{1} >$. In this case,  the Hochschild complex obtained by applying the funtor $\Hom_{A^e}(-, A)$ to  Bardzell's resolution is the following   
\begin{multline*} 0 \to  \Hom_{A^e} (A\otimes \k Q_0 \otimes A , A)   \stackrel{d^{1}}{\to}   \Hom_{A^e} (A\otimes \k Q_1 \otimes A , A)  \\
  \stackrel{d^{2}}{\to}   \Hom_{A^e} (A\otimes \k I  \otimes A , A)   \to  \dots    
  \to  \Hom_{A^e} (A\otimes \k AP_{n-1}\otimes A , A)  \\  \stackrel{d^{n}}{\to}  \Hom_{A^e} (A\otimes \k AP_n \otimes A , A)   \to  \dots  \end{multline*}
  where $AP_n = \{  (\alpha_{1}\dots\alpha_r)^{n-1}\alpha_{1}  \}$ for $n \geq 2$. A direct calculation shows  that the maps $d^n = \Hom_{A^e}(d_n, A)$ are zero for $n \geq 2$. To describe the map $d^1$, observe that the $A^e$-linear maps $h_i, g_i : A\otimes \k Q_0 \otimes A \to A $, $1 \leq i \leq r$, defined by  
 \begin{align*}
 h_{i}( 1\otimes e_j \otimes 1) & =  \delta_{ij} \ e_i \\
 g_{i}(1\otimes e_j \otimes 1) &  =  \delta_{ij} \ \alpha_i \dots \alpha_r \alpha_1 \dots \alpha_{i-1},
 \end{align*}
 with $\delta_{ij}$ the Kronecker delta,  are generators of the vector space $\Hom_{A^e} (A\otimes \k Q_0 \otimes A , A) $ and $\Ker d^1= \langle \sum_{i=1}^rh_i, \  \sum_{i=1}^rg_i \rangle$.  So $\dim_\k \Im d^1 = 2r-2$. Finally    a straightforward computation shows that  $\dim_\k \HH^n (A)  =1$ for each $n \geq 1$, and  the class of the $A^e$-linear map $f_{n}: A \otimes \k AP_n \otimes A \to A$  defined  by 
 $$f_{n}( 1\otimes (\alpha_{1}\dots\alpha_r)^{n-1}\alpha_1 \otimes 1) =  \alpha_1 $$ is a generator of $ \HH^n(A)$.
 
 Now we will compute the Lie bracket and the cup product using the previous  formulae that arose from the comparison morphisms.
 For this,  we observe that,  for $n \geq 3$,   $\Sub( (\alpha_{1}\dots\alpha_r)^{n-1}\alpha_1 ) = \{ \psi_1, \psi_2\}$ with $\psi_1 = \psi_2 =  (\alpha_{1}\dots\alpha_r)^{n-2}\alpha_1$ because  $(\alpha_{1}\dots\alpha_r)^{n-1}\alpha_1 =  (\alpha_{1}\dots\alpha_r)\psi_2 = \psi_1(\alpha_{2}\dots\alpha_r\alpha_1) $.  By Remark \ref{obs1}(2)  we have that 
 $$F_n(1 \otimes (\alpha_{1}\dots\alpha_r)^{n-1}\alpha_1  \otimes 1) = 1 \otimes L  F_{n-1}(1 \otimes  (\alpha_{1}\dots\alpha_r)^{n-2}\alpha_1 \otimes 1)$$ where we denote   $L = \alpha_{1}\dots\alpha_r$.  Thus, if we continue applying this remark  we get that
   \begin{align*}  F_{n}(1 \otimes (& \alpha_{1}\dots\alpha_r)^{n-1}\alpha_1 \otimes 1)  =   1 \otimes L^{\otimes n-3} \otimes  LF_2(1 \otimes \alpha_1\dots \alpha_r\alpha_1 \otimes 1)  \\ & = 
 \sum_{j=1}^{r-1} 1 \otimes L^{\otimes n-2}  \otimes \alpha_1\dots \alpha_j \otimes \alpha_{j+1} \otimes \alpha_{j+2} \dots \alpha_r\alpha_1  
\  +  \  1 \otimes L^{\otimes n-1}  \otimes \alpha_{1} \otimes 1\end{align*}
and this formula holds for any $n \geq 1$. 
By definition of the morphism $G$, we have that
\begin{align*}  
& G_1 (1 \otimes \alpha \otimes 1 )   =   1 \otimes \alpha \otimes 1, \quad  \mbox{for any $\alpha \in Q$, and for $n \geq 2$,}\\
&  G_n(1 \otimes L^{\otimes n-2} \otimes \alpha_1\dots\alpha_j \otimes \alpha_{j+1}  \otimes  1)     =      0,   \mbox{ \quad if $1 \leq j \leq r-1$,} \\
& G_n(1 \otimes L^{\otimes n-1} \otimes \alpha_1\dots \alpha_j \otimes 1 )   =   1 \otimes  (\alpha_{1}\dots\alpha_r)^{n-1}\alpha_1  \otimes \alpha_2 \dots \alpha_j,  \mbox{ \quad if $1 \leq j \leq r$.} 
\end{align*}
So, taking into account the above considerations, we get that if $1 \leq j \leq r-1$, 
\begin{align*} 
G^{n}(f_{n}) \circ_{i} G^{m}(f_{m}) ( 1 \otimes L^{n+m-3} \otimes  \alpha_1\dots \alpha_j \otimes \alpha_{j+1} \otimes \alpha_{j+2} \dots \alpha_r\alpha_1 ) & =   0  \\
& \mbox{for $1 \leq i \leq n$, }  \\ 
 G^{n}(f_{n}) \cup G^{m}(f_{m}) ( 1 \otimes L^{n+m-2} \otimes  \alpha_1\dots \alpha_j \otimes \alpha_{j+1} \otimes \alpha_{j+2} \dots \alpha_r\alpha_1 )&  =   0 \end{align*}
 and \begin{align*} G^{n}( f_{n}) \circ_{i}G^{m}( f_{m})( 1 \otimes L^{\otimes n+m-2} \otimes   \alpha_{1} \otimes 1) & =   \alpha_1\\
 G^{n}( f_{n}) \cup G^{m}( f_{m})( 1 \otimes L^{\otimes n+m-1} \otimes   \alpha_{1} \otimes 1) & =    0. \end{align*}
So    
$$ f_{n} \cup f_{m}  =    F^{n+m}(  G^{n}(f_{n}) \cup  G^{m}(f_{m})) =  0, $$  
$$  f_{n} \circ_i f_{m}  =    F^{n+m-1}(  G^{n}(f_{n}) \circ_i G^{m}(f_{m})) = f_{n+m-1}, \mbox{ and} $$  $$ f_{n} \circ f_{m} = \sum_{i=1}^{n}(-1)^{(i-1)(m-1)}f_{n+m-1}. $$  Finally we get that,  for any $n,m \geq 1$, 
$$[ f_{n}, f_{m}] =
\begin{cases}
(n-m)f_{n+m-1}, & \hbox{if $n, m$ are odd;} \\
(n-1)f_{n+m-1}, & \hbox{if $n$ is even and $m$ is odd;} \\
0, & \hbox{if $n, m$ are even.} \\
\end{cases}$$
If we only consider odd degrees, that is,  $$ \HH^{odd}(A) = \bigoplus_{\substack{q \geq 0,} } \HH^{2q+ 1}(A),$$   then  $ \HH^{odd}(A)$ is isomorphic, as Lie algebras, to the infinite dimensional Witt algebra.
\end{example}

In order to get more general results concerning the cup product and the Lie bracket for the Hochschild cohomology groups of monomial algebras, we will start by studying more carefully the morphism ${\mathbf F}$.  This description will imply that  ${\mathbf G} \circ {\mathbf F }= \Id_{\Ap}$.

\bigskip

\begin{lemma}\label{sumaf} Let $w \in AP_n$, $n \geq 2$ and $\psi_i \in AP_{n-i}$ given by
\begin{align*} w & =  L(\psi_1)\psi_1 \\
\psi_i & = L(\psi_{i+1})\psi_{i+1}, \quad \mbox{ $i = 1, \dots,
n-2$,}
\end{align*}
that is, 
$w = L(\psi_1) L(\psi_2) \dots L(\psi_{n-1}) \psi_{n-1}$.  Then 
\begin{align*}F_n(1\otimes w\otimes 1)
 = & \sum_{(a_1, a_2, \dots, a_n , b) \in \mathcal{K}} 1 \otimes
a_1 \otimes a_2 \otimes \dots \otimes a_n \otimes b \\
  &  +     1 \otimes L(\psi_1) \otimes L(\psi_2)
\otimes \dots \otimes L(\psi_{n-1}) \otimes \psi_{n-1}
\otimes 1  \end{align*} where $\mathcal{K}$ is a subset of 
$\{  (a_1, a_2, \dots, a_n , b):   a_1,   \dots,  a_n,  b \in \P : \
\vert b
 \vert > 0,  w = a_1\dots a_nb \} .$
\end{lemma}

\begin{proof} The result will be proved using induction on $n$.  If $n=2$ the result is clear from Remark \ref{obs1}(1). Assume $n>2$ and let 
$w \in AP_n$ with $\Sub(w) = \{\zeta_{1}, \dots,  \zeta_{m}  \} $, where $\zeta_{m} = \psi_1$.  Then 
\begin{align*}F_n(1\otimes w\otimes 1)   =  \sum_{i= 1}^{m-2} 1\otimes
 L_{i+1}F_{n-1}(1\otimes \zeta_{i+1} \otimes
1)R_{i+1} \ 
 +  \ 1 \otimes L(\psi_1)F_{n-1}(1 \otimes \psi_1 \otimes 1).\end{align*} 
By induction hypothesis we have that 
\begin{align*}F_{n-1}(1\otimes \psi_1 \otimes 1)
 = & \sum_{(a_2, \dots, a_n , b) \in \mathcal{K'}} 1 \otimes
a_2 \otimes a_3 \otimes \dots \otimes a_n \otimes b \\
  &  +     1 \otimes L(\psi_2) \otimes L(\psi_3)
\otimes \dots \otimes L(\psi_{n-1}) \otimes \psi_{n-1}
\otimes 1  \end{align*} where $\mathcal{K'}$ is a subset of 
$\{  (a_2, \dots, a_n , b):   a_2,   \dots,  a_n,  b \in \P : \
\vert b
 \vert > 0,  \psi_1 = a_2\dots a_nb \} $ and hence the result follows since $w  =  L(\psi_1)\psi_1$.
\end{proof}

\begin{prop} The comparison morphisms defined in Subsections \ref{F} and \ref{G} respectively,  satisfy the equality  ${\mathbf G} \circ {\mathbf F} = \Id_{\Ap A}$. \end{prop}
\begin{proof}
A direct computation shows that
\begin{align*} G_0 \circ F_0(1 \otimes e \otimes 1)&  =  G_0(e
\otimes 1) = e \otimes 1 \otimes 1 = 1 \otimes e \otimes 1, \mbox{ \
and } \\ G_1 \circ F_1(1 \otimes \alpha \otimes 1)&  =   G_1 (1
\otimes \alpha \otimes 1) = 1 \otimes \alpha \otimes 1.
\end{align*}
If $w \in AP_n$ with $n \geq 2$, the previous lemma says that
\begin{align*} G_n \circ F_n(1 \otimes w \otimes 1)  = &
\sum_{(a_1, a_2, \dots, a_n , b) \in \mathcal{K}} G_n (1 \otimes a_1
\otimes a_2 \otimes \dots \otimes a_n \otimes b)   
\\ & +  G_n( 1
\otimes L(\psi_1) \otimes L(\psi_2) \otimes \dots \otimes
L(\psi_{n-1}) \otimes \psi_{n-1} \otimes 1 ). \end{align*}
If  $w = a_1\dots a_nb$ with $\vert b \vert > 0$, it is clear that
$\chi(a_1, \dots, a_n) = \emptyset$, since otherwise there would exist
$z \in AP_n$ a divisor of $a_1\dots a_n$ and hence $z$ and $w$ would belong to  $AP_n$  with $z$ a proper divisor of 
$w$, a contradiction. Thus
\begin{equation*} \sum_{(a_1, a_2, \dots, a_n , b) \in \mathcal{K}} G_n (1 \otimes a_1 \otimes a_2 \otimes \dots \otimes
a_n \otimes b)  = 0. \end{equation*} 
On the other hand, if $w =
w^{op}(q^1, \dots , q^{n-1})$ then $\psi_{i-1}   =  (q^{i}, \dots, q^{n-1})$ and $\psi_{i-1} = L(\psi_{i}) L(\psi_{i+1}) \psi_{i+1}$.  Hence  the  $n$-sequence $(L(\psi_1), \dots,
L(\psi_{n-1}), \psi_{n-1})$ is good because
$q^i$ divides $L(\psi_i)L(\psi_{i+1})$.  Thus 
\begin{equation*}  G_n( 1
\otimes L(\psi_1) \otimes L(\psi_2) \otimes \dots \otimes
L(\psi_{n-1}) \otimes \psi_{n-1} \otimes 1 ) = 1 \otimes w \otimes
1\end{equation*}
since
 $w = L(\psi_1) L(\psi_2)\dots
L(\psi_{n-1}) \psi_{n-1}.$ 
\end{proof}

\subsection{ The module structure of $\HH^*(A)$ over the Lie algebra $\HH^1(A)$}\label{modulo}

The first cohomology group $\HH^1(A)$ is a Lie algebra 
and the Lie bracket induces a module structure of $\HH^*(A)$ over the Lie algebra $\HH^1(A)$. 
For any $f \in \Hom_{A^e}(A \otimes \k AP_n \otimes A, A)$ and $\delta \in \Hom_{A^e}(A \otimes \k Q_1 \otimes A,A)$, 
 the description of $[\delta, f]$ involves the following computations             
$$ F^n
(G^n(f) \circ_i G^1(\delta)) \quad ( i =
1, \dots, n) \  \mbox{and}  \   F^n (G^1(\delta) \circ_1
G^n(f)). $$
For any $w \in
AP_n$, using Lemma \ref{sumaf} we get that
\begin{align*} G^n(f) \circ_i & G^1(\delta)(F_{n}( 1 \otimes w \otimes 1))
     \\   = & \sum_{(a_1, a_2,
\dots, a_n , b) \in \mathcal{K}}G^n(f)(1 \otimes a_1 \otimes \dots \otimes
G^1(\delta)(1 \otimes a_i \otimes 1 ) \otimes \dots \otimes a_{n} \otimes b)
\\  &  + \quad  G^n(f) (1 \otimes  L(\psi_1) \otimes \dots \otimes G^1(\delta)(1 \otimes L(\psi_i) \otimes 1) \otimes \dots
\otimes \psi_{n-1} \otimes 1). 
\end{align*}
and 
\begin{equation*}
\begin{split}
G^1(\delta) \circ
G^n(f)(F_n(1 \otimes w \otimes 1))  & =   
G^1(\delta)\left(1 \otimes f(G_n\circ F_n (1 \otimes w \otimes 1) ) \otimes 1 \right) \\
& =  G^1(\delta)\left(1 \otimes f(1 \otimes w \otimes 1)\otimes 1 \right)
 \end{split}
 \end{equation*}
where the second equality follows from the fact that  $G_n\circ F_n =
\Id_{AP_n}$.

In the particular case of a  
monomial algebra $A$ that satisfies the  property \linebreak
$\dim_\k e_i A
e_j = 1$ if there exists $\alpha: i \to j \in Q_1$, the group 
 $\HH^1(A)$ is generated by the set $(Q_1//Q_1)= \{ (\alpha, \alpha), \alpha \in Q_1\}$, see {\cite{CS}.}
The map $\delta_\alpha: A \otimes \k Q_1 \otimes A \to A$ defined by 
\[ \delta_\alpha ( 1 \otimes \beta \otimes 1) = 
\begin{cases}
\alpha, & \hbox{if $ \beta = \alpha$;} \\
0, & \hbox{otherwise}
\end{cases}\] 
represents the generator $(\alpha, \alpha)$ of $ \HH^1(A)$.  
If $v= \alpha_1\dots \alpha_s$ , $\alpha_i \in Q_1$ we denote by $C(\alpha, v) =
\vert \{ i : \alpha = \alpha_i, \ i = 1, \dots,  s \} \vert $, and we extend by linearity to a map  $C(\alpha, -) : A \to \k$.
             
By definition, $G^1(\delta_\alpha)(1 \otimes w \otimes 1)= C(\alpha, w)w$ and a direct computation shows that
$$ G^n(f) \circ_i
 G^1(\delta_\alpha)(F_{n} (1 \otimes  w \otimes  1))  =  C(\alpha,
L(\psi_i)) f(1 \otimes w \otimes 1 ).$$
Hence  $$F^n (G^n(f) \circ G^1(\delta_\alpha))(1 \otimes w \otimes 1) =
 C(\alpha, w) f(1 \otimes w \otimes 1)$$ and
 $$ G^1(\delta_\alpha) \circ  G^n(f) (F_{n} (1 \otimes  w \otimes  1)) =  C(\alpha,
f(1 \otimes w \otimes 1)) f(1 \otimes w \otimes 1), $$
so
$$[\delta_\alpha, f] (1 \otimes w \otimes 1 ) = \left(C(\alpha, w)  -  C(\alpha,
f(1 \otimes w \otimes 1)) \right) f(1 \otimes w \otimes 1).$$

 This result has also been proved in \cite{MSA}.

\subsection{The cup product in  $\oplus_{n \geq 0} \HH^{2n}(A)$}

Using our formulas we will show that the cup product restricted to even degrees of the Hochschild cohomology has a very simple description.

\begin{teo}
Let $f \in HH^{2n}(A)$, $g \in \HH^{2m}(A)$.  Then 
$$f \cup g (1 \otimes w \otimes 1) = f(1 \otimes w(p_1, \dots, p_{2n-1}) \otimes 1)a g(1 \otimes w^{op}(q^{2n+1}, \dots, q^{2n+2m-1})\otimes 1),$$
where $w = w(p_1, \dots, p_{2n-1}) a w^{op}(q^{2n+1}, \dots, q^{2n+2m-1})$.
\end{teo}

\begin{proof}
For any  $f \in \Hom_{A^e}(A \otimes \k AP_{2n} \otimes A, A)$ and $g \in \Hom_{A^e}(A \otimes \k AP_{2m} \otimes A, A)$
we have already seen that the cup product $f \cup g$ is given by
$$F^{2n+2m}(  G^{2n}(f) \cup  G^{2m}(g))$$ and for any  $w \in AP_{2n+2m}$  we can compute it using 
Lemma \ref{sumaf} as follows:
\begin{align*} 
F^{2n+2m}(&  G^{2n}(f) \cup   G^{2m}(g))(1 \otimes w \otimes 1) =  \quad  G^{2n}(f) \cup
G^{2m}(g)(F_{2n+2m}(1 \otimes w \otimes 1))\\  
 =    & \sum_{(a_1, a_2, \dots, a_{2n+2m}
, b) \in \mathcal{K}}G^{2n}(f)( 1 \otimes a_1 \otimes \dots \otimes
a_{2n} \otimes 1) \\
& \quad  \times G^{2m}(g)(1 \otimes a_{2n+1} \otimes \dots \otimes a_{2n+2m} \otimes b )  \\ & + 
\quad  G^{2n}(f)(1 \otimes L(\psi_1) \otimes \dots \otimes L(\psi_{2n})\otimes 1
) \\
& \quad  \times G^{2m}(g)(1 \otimes L(\psi_{2n+1}) \otimes \dots \otimes \psi_{2n+2m-1} \otimes 1).
\end{align*}
All the terms in the previous sum that contain the path $b$ vanish, since otherwise
$$G^{2m}(g)(1 \otimes a_{2n+1} \otimes \dots \otimes
a_{2n+2m} \otimes 1) \neq 0$$
and in this case there should exist  $u \in AP_{2m}$ a divisor of $a_{2n+1}
\dots a_{2n+2m}$. Since \linebreak $\vert b\vert > 0$,  we have that $t(u) \leq t(a_{2n+2m}) <
t(\psi_{2n+2m-1})$. Applying Lemma \ref{A2}(ii) to the concatenations $u$ and 
$$w^{op}(q^{2n+1}, \dots,
q^{2n+2m-1}) = L(\psi_{2n+1})  \dots L(\psi_{2n+2m-1}) \psi_{2n+2m-1}$$ in $AP_{2m}$  we know that there exists $z \in AP_{2m+1}$
with $s(z) = s(u)$.  We can picture this situation as follows: 
\[ \xymatrix{ & &  \ar@{|-|}[rrrrrrrr]^{w^{op}(q^{2n+1}, \dots, q^{2n+2m-1})} & & &  & & & & &   \\
 &   \ar@{|-|}[rrrrrrr]^{u} &  & & & & & & & &  \\
 &  \ar@{|-|}[rrrrrrrr]^{z} &  &  & & & & & &  &  .\\
   } \]   
Now we are going to compare $s(u)$ and $s(q^{2n})$. Assume $s(q^{2n}) < s(u)$; then \linebreak $z
\in AP_{2m+1}$ should be a proper divisor of $w^{op}(q^{2n}, \dots, q^{2n+2m-1})
\in AP_{2m+1}$, a contradiction. Then  $s(u) \leq
s(q^{2n})$, and since $ G^{2n}(f)(1 \otimes a_1 \otimes \dots \otimes a_{2n} \otimes 1) \neq 0$
there should exist $v \in AP_{2n}$ a divisor of $a_1 \dots  a_{2n}$  and hence 
$$s(a_1)\leq s(v) \quad  \mbox{ and } \quad  t(v) \leq t(a_{2n}) \leq s(u).$$ From \cite{B}*{Lemma 3.1}
we deduce that 
$$w = w^{op}(q^1, \dots, q^{2n+2m-1}) = w(p_1, \dots,
p_{2n+2m-1})$$ and $s(q^{2n}) < t(p_{2n-1})$, 
so
$$ s(p_1) = s(a_1) \leq s(v) \leq t(v) \leq s(u) \leq s(q^{2n}) < t(p_{2n-1}).$$ Hence  $v \in AP_{2n}$ is a proper divisor of $w(p_1, \dots, p_{2n-1}) \in AP_{2n}$, a contradiction.

Now we can conclude that 
 \begin{align*}
 F^{2n+2m} ( G^{2n}&(f) \cup G^{2m}(g))(1 \otimes w \otimes 1)  \\
  = & G^{2n}(f)(1 \otimes L(\psi_1) \otimes
\dots \otimes L(\psi_{2n})\otimes 1 ) \\
& \times G^{2m}(g)(1 \otimes L(\psi_{2n+1}) \otimes \dots \otimes
\psi_{2n+2m-1}) \otimes 1).
\end{align*}
Using that  $L(\psi_{2n+1})  \dots L(\psi_{2n+2m-1}) \psi_{2n+2m-1} =
w^{op}(q^{2n+1}, \dots, q^{2n+2m-1})\in AP_{2m}$ we have that 
$$G^{2m}(g)(1 \otimes L(\psi_{2n+1}) \otimes \dots \otimes
\psi_{2n+2m-1} \otimes 1) = g ( 1 \otimes w^{op}(q^{2n+1}, \dots, q^{2n+2m-1}) \otimes 1).$$  On the other hand, \cite{B}*{Lemma 3.1} says that
$t(p_{2n-1}) \leq t(q^{2n-1}) <s(q^{2n+1})$ and then $w(p_1, \dots, p_{2n-1}) \in \chi(L(\psi_1),\dots ,L(\psi_{2n}))$,
and since $s(p_1) = s(L(\psi_1))$,  its starting point must be minimal with respect to all the concatenations considered in the previous set. Hence
$$ G^{2n}(f)(1 \otimes L(\psi_1) \otimes
\dots \otimes L(\psi_{2n}) \otimes 1 )   = f(1 \otimes w(p_1, \dots, p_{2n-1}) \otimes a).$$ \end{proof}


\begin{bibdiv}
\begin{biblist}


\bib{ACT}{article}{
   author={Ames, Guillermo},
   author={Cagliero, Leandro},
   author={Tirao, Paulo},
   title={Comparison morphisms and the Hochschild cohomology ring of
   truncated quiver algebras},
   journal={J. Algebra},
   volume={322},
   date={2009},
   number={5},
   pages={1466--1497},
}

\bib{ASS}{book}{
   author={Assem, Ibrahim},
   author={Simson, Daniel},
   author={Skowro{\'n}ski, Andrzej},
   title={Elements of the representation theory of associative algebras.
   Vol. 1},
   series={London Mathematical Society Student Texts},
   volume={65},
   note={Techniques of representation theory},
   publisher={Cambridge University Press, Cambridge},
   date={2006},
   pages={x+458},
}
	


\bib{B}{article}{
   author={Bardzell, Michael J.},
   title={The alternating syzygy behavior of monomial algebras},
   journal={J. Algebra},
   volume={188},
   date={1997},
   number={1},
   pages={69--89},
}



\bib{BGSS}{article}{
   author={Buchweitz, Ragnar-Olaf},
   author={Green, Edward L.},
   author={Snashall, Nicole},
   author={Solberg, {\O}yvind},
   title={Multiplicative structures for Koszul algebras},
   journal={Q. J. Math.},
   volume={59},
   date={2008},
   number={4},
   pages={441--454},
 }


\bib{CE}{book}{
   author={Cartan, Henri},
   author={Eilenberg, Samuel},
   title={Homological algebra},
   series={Princeton Landmarks in Mathematics},
   note={With an appendix by David A. Buchsbaum;
   Reprint of the 1956 original},
   publisher={Princeton University Press, Princeton, NJ},
   date={1999},
   pages={xvi+390},
}



\bib{C3}{article}{
   author={Cibils, Claude},
   title={Hochschild cohomology algebra of radical square zero algebras},
   conference={
      title={Algebras and modules, II},
      address={Geiranger},
      date={1996},
   },
   book={
      series={CMS Conf. Proc.},
      volume={24},
      publisher={Amer. Math. Soc., Providence, RI},
   },
   date={1998},
   pages={93--101},}





\bib{G}{article}{
   author={Gerstenhaber, Murray},
   title={The cohomology structure of an associative ring},
   journal={Ann. of Math. (2)},
   volume={78},
   date={1963},
   pages={267--288},
}
		
\bib{GS}{article}{
   author={Gerstenhaber, Murray},
   author={Schack, Samuel D.},
   title={Algebraic cohomology and deformation theory},
   conference={
      title={Deformation theory of algebras and structures and applications},
      address={Il Ciocco},
      date={1986},
   },
   book={
      series={NATO Adv. Sci. Inst. Ser. C Math. Phys. Sci.},
      volume={247},
      publisher={Kluwer Acad. Publ., Dordrecht},
   },
   date={1988},
   pages={11--264},
}
		

\bib{Ho}{article}{
   author={Hochschild, G.},
   title={On the cohomology groups of an associative algebra},
   journal={Ann. of Math. (2)},
   volume={46},
   date={1945},
   pages={58--67},
}	
	

\bib{NW}{article}{
   author={Negron, Cris},
   author={Witherspoon, Sarah},
   title={An alternate approach to the Lie bracket on Hochschild cohomology},
   journal={Homology Homotopy Appl.},
   volume={18},
   date={2016},
   number={1},
   pages={265--285},
}



\bib{RR}{article}{
 author={Redondo, Mar{\'{\i}}a Julia},
 author={Rom{\'a}n, Lucrecia},
  title={Gerstenhaber algebra structure on the Hochschild cohomology of quadratic string algebra},
  journal={Algebr Represent Theor},
  date={2017},
 doi={10.1007/s10468-017-9704-1} }	

	

\bib{SS}{article}{
   author={Schwede, Stefan},
   title={An exact sequence interpretation of the Lie bracket in Hochschild
   cohomology},
   journal={J. Reine Angew. Math.},
   volume={498},
      date={1998},
   pages={153--172},
}

\bib{CS}{article}{
   author={Strametz, Claudia},
   title={The Lie algebra structure on the first Hochschild
cohomology group of a monomial algebra},
   journal={J.  Algebra Appl.},
   volume={5},
   date={2006},
   number={3},
   pages={245--270},
}

\bib{MSA}{article}{
   author={Su\'arez-\'Alvarez, Mariano},
   title={A little bit of extra functoriality for Ext and the computation of
   the Gerstenhaber bracket},
   journal={J. Pure Appl. Algebra},
   volume={221},
   date={2017},
   number={8},
   pages={1981--1998},
}







\end{biblist}
\end{bibdiv}
\end{document}